\documentclass[11pt]{article}
\usepackage{geometry}                
\usepackage{graphicx}    
\usepackage{amssymb,amsmath,epsfig}
\usepackage{scalefnt}
\usepackage{MnSymbol}
\usepackage[all]{xypic}
\usepackage{slashed}
\usepackage{extarrows}
\usepackage{stmaryrd}
\usepackage{lettrine}
\usepackage{booktabs}
\usepackage{paralist}
\usepackage{subfig}
\newtheorem{lemma}{Lemma}[section] 
\newtheorem{proposition}[lemma]{Proposition}
\newtheorem{example}[lemma]{Example}

\newtheorem{cor}[lemma]{Corollary}

\newtheorem{definition}[lemma]{Definition}

\newcommand*{\doublenabla}{%
  \nabla\mkern-12mu\nabla
}   

\newcommand{\CM}{\hbox{{$\mathcal M$}}}

\newcommand{\C}{\mathbb{C}}
\newcommand{\R}{\mathbb{R}}

\newcommand{\Z}{\mathbb{Z}}

\newcommand{\ev}{\mathrm{ev}}
\newcommand{\coev}{\mathrm{coev}}

\newcommand{\extd}{\mathrm{d}}

\newcommand{\tens}{\mathop{{\otimes}}}
\newcommand{\la}{{\triangleright}}
\newcommand{\ra}{{\triangleleft}}

\newcommand{\id}{\mathrm{id}}

\newcommand{\<}{\langle}
\renewcommand{\>}{\rangle}

\newcommand{\cX}{\mathfrak{X}}

\title{Noncommutative geodesics and the KSGNS construction}
\author{Edwin Beggs \\
College of Science, Swansea University, Wales}
\date{}                                           

\begin{document}

\maketitle

\abstract{We study geodesics in noncommutative geometry by means of
 bimodule connections and completely positive maps using the Kasparov, Stinespring, Gel'fand, Na\u\i mark \& Segal (KSGNS) construction. This is motivated from classical geometry, and we also consider examples on the algebras $M_2(\C)$ and $C(\Z_n)$, though restricting to classical time $t\in\R$. On the way we have to consider the reality of a noncommutative vector field, and for this we propose a definition depending on a state on the algebra.}

\section{Introduction}
In classical geometry we frequently consider flows on manifolds due to vector fields (e.g.\ Morse theory) or vector fields as velocities along paths (e.g.\ geodesics). The definition of noncommutative vector field in generality was given by Borowiec \cite{Borowiec} in terms of a generalised derivation (a Cartan pair), and used to used to define cases of Lie brackets by Jara \& Llena
 \cite{JarLle}. Vector fields on Hopf algebras had been considered in \cite{AscSch,Ma:cla,LychBraid}. However, 
 it has been difficult to apply noncommutative vector fields to the two classical applications above.

One fundamental decision is just what sort of maps to take between $C^*$-algebras or dense subalgebras of these. For various purposes the class of maps has been extended beyond $*$-algebra maps. For example Connes \& Higson \cite{conHig} introduced asymptotic $*$-algebra maps (called asymptotic morphisms) for $E$-theory, and these maps were also used by D\~ad\~arlat \cite{DadAsy} and Manuilov \& Thomsen \cite{ManThom} for noncommutative shape theory. Connes introduced the idea of correspondences (bimodules) to study von Neumann algebras \cite{Connes}. Completely positive maps received attention from many authors, several of which (Kasparov, Stinespring, Gel'fand, Na\u\i mark \& Segal) appear in the name of the  KSGNS construction. For this construction and the theory of Hilbert $C^*$-bimodules we refer to the textbook \cite{Lance}. 

The purpose of this paper is to apply the KSGNS construction to noncommutative geodesics. The natural interpretation of these geodesics will be as paths in the state space of the algebra, i.e.\ completely positive maps from the algebra to $C^\infty(\R)$.
 Classically, evaluation at a point of a space $X$ is a state on $C(X)$, and moving the point along a path moves the state.
   Given the general setting of the noncommutatrive construction, it is likely that this restriction to the `time algebra' being $C^\infty(\R)$ is unnecessary, but here we shall stick to `real commutative time'.

The KSGNS construction represents a completely positive map between $C^*$-algebras as $\phi(a)=\<ma,\overline{m}\>$ in terms of an element $m$ of a Hilbert $C^*$-bimodule. The noncommutative theory of connections on bimodules has been studied for some time. We simply put these ingredients together, and look at examples. The critical result is that geodesics in classical differential geometry are precisely recovered as a special case. On the way we require a constraint on the vector fields involved which in the classical case simply amounts to the reality of the vector field.
Like other ideas generalising classical geometry, the reality of a vector field can only be defined with hindsight given by a sufficient number of theory and examples, so the `definition' here is merely a trial one.

The KSGNS construction is also well adapted to dealing with quantum theory, indeed it contains the usual theory of Hilbert spaces and observables (with some extension to unbounded operators). 
Paths in classical differential geometry have, at a given time, a precise position and a precise velocity. It is somewhat obvious that this idea will have to be modified in quantum theory, as there position and velocity (or rather momentum) obey the Heisenberg uncertainty principle. 
But why should parallel transport or geodesics make sense in quantum  theory?
 The answer is simply that we can observe geodesic motion in the real world, so if the real world is governed by quantum theory, then to some extent geodesics must still make sense. 

Quantum theory, by quantising the stress-energy tensor source term for gravity in General Relativity, \textit{de facto} quantises geometry, and at a scale conceivably much larger than the Planck length. There is no reason to expect that observations of quantum gravity will necessarily first take place in 
measurements of momentum eigenstates, in other words, in the normal domain of the perturbation theory solution methods of quantum field theory. 
Given the local nature of gravitational fields, it is likely that measurements of position will be involved.
To back up any such observations it would be necessary to have a theory allowing the calculation of positions in quantum gravity, and as quantum gravity may well manifest itself, at least to `first order', as noncommutative geometry, that may mean a physical theory of paths or world lines in noncommutative geometry.

As the reader will see, there is often great flexibility in extending classical ideas to noncommutative geometry. Saying that a particular way is \textit{the} way is not something to say lightly. The purpose of this paper is simply to show that \textit{a} way of addressing geodesics in noncommutative geometry exists, and that it can be applied to many examples. It proposes that
for a $C^\infty(\R)$-$A$-bimodule $M$
 the equation $\doublenabla(\sigma_M)=0$ is a reasonable and calculable extension of the equations for vector fields in classical geodesics, and that $\nabla_M(m)=0$ for $m\in M$ in the KSGNS construction above gives the corresponding time evolution on the state space.
 
There are other matters which we do not address, such as the constant speed of a classical geodesic and geodesic deviation. Addressing such matters, if it were possible, to avoid very long individual proofs would likely require methods for handling derivatives of maps (here denoted $\doublenabla$) compatible with tensor products, such as the extension of the monoidal DG category ${}_A\mathcal{G}{}_A$ in \cite{BegMa:bia} to mixed bimodules as a coloured monoidal DG category $\mathcal{G}$.
This paper is phrased in terms of bimodule connections as its basic object, and classically we might take the Levi Civita connection once we have a metric. By starting with bimodule connections we give a potentially more general discussion, and avoid another problem about just what connection to take for what sort of noncommutative metric. Further, it may be the case that quantum theory may have use for more general connections, e.g., in the presence of particle creation or annihilation in quantum mechanics we might well have positive functions with varying normalisation. As a possible example related to geometry, but well beyond the scope of this paper, Hawking radiation \cite{HawRad} produces particles at the expense of the intrinsic energy of space-time -- which is what the negative energy states disappearing into the black hole and reducing its mass effectively amount to.

In \cite{AppSto} there is a construction of quantum stochastic parallel transport processes
 which are possibly related to the methods in this paper.
The numerical equation solving and graphics were done on Mathematica, and some code for this is given in the Appendix.

\section{Preliminaries} \label{prelim}

Suppose that $A$ is a unital possibly noncommutative algebra over the field $\mathbb{C}$ (taking this choice to link with $C^*$-algebras later). 
We think of $A$ as the $\C$ valued functions on a hypothetical {\em{noncommutative manifold}}. An $A$-module will correspond to a vector bundle on this hypothetical manifold. 
The tensor product of vector bundles corresponds to taking $E\tens_A N$ where $E$ is a right and $N$ a left $A$-module. Here $E\tens_A N$ has elements $e\tens n$ for $e\in E$ and $n\in N$ where we set $e.a\tens n=e\tens a.n$ for all $a\in A$.

\begin{definition} A first order differential calculus $(\Omega^1,\extd)$ over $A$ means

{\rm (1)}\  $\Omega^1_A$ an $A$-bimodule.

{\rm (2)}\  A linear map $\extd:A\to \Omega^1_A$ (the exterior
derivative) with $\extd(a\,b)=(\extd a)\,b + a\,\extd b$\\
\phantom{{\rm (2)}\quad } for all $a,b\in A$. 

{\rm (3)}\  $\Omega^1_A={\rm span}\{a\,\extd b\mid a,b\in A\}$ (the surjectivity condition).

\noindent 
A (right) vector field on $A$, notation $v\in\cX_A$, is a right $A$-module map $v:\Omega^1_A\to A$. 

\end{definition}

\begin{example}
For the usual calculus on $\R^n$,  the algebra $A=C^\infty(\R^n)$  (functions on $\R^n$ which are differentiable infinitely 
 many times, and complex valued as noted earlier) has $\Omega^1_A$ the usual 1-forms on $\R^n$, i.e.\ $\xi_i\,\extd x^i$ (sum over $i$) for coordinates $x^1,\dots,x^n$ and $\xi_i\in C^\infty(\R^n)$. We will just call this $\Omega^1(\R^n)$ to avoid writing $C^\infty(\R^n)$ as a subscript. 

For the usual calculus on $\R^n$, the vector fields  $\cX(\R^n)$ 
are of the form $v=v^i\,\frac{\partial}{\partial x^i}$. They are maps from $\Omega^1(\R^n)$ to
$A=C^\infty(\R^n)$ via the evaluation $\ev:\cX(\R^n) \tens_A \Omega^1(\R^n)\to C^\infty(\R^n)$ which is $\ev(v^i\,\frac{\partial}{\partial x^i}\tens \xi_j\,\extd x^j)=v^i\,\xi_i$. 
We also have a dual basis of vector fields, which is expressed as a single element
 $\coev(1)=\extd x^i\tens \frac{\partial}{\partial x^i}\in \Omega^1(\R^n)  \tens_A \cX(\R^n)$
 or more categorically by the coevaluation bimodule map $\coev: A\to \Omega^1_A\tens_A \cX_A$. This has the property that
 $(\id\tens\ev)(\coev(1)\tens \xi)=\xi$ for all $\xi\in\Omega^1(\R^n)$, which is easily verified by
\[
(\id\tens\ev)\big( \extd x^i\tens \tfrac{\partial}{\partial x^i} \tens \xi_j\,\extd x^j\big)= \extd x^i\,\delta_{i,j}\, \xi_j=\xi_j\,\extd x^j\ .\quad\diamond
\]
  \end{example}
 
 We now give two examples of noncommutative calculi, the first on a noncommutative algebra and the second a noncommutative calcus on a commutative algebra.

\begin{example} \label{matcalcex}
Set $A=M_2(\C)$, the 2 by 2 complex matrices. This is given a calculus 
 where $\Omega^1_A$ is freely generated by two central generators $s^1$ and $s^2$, with
\[
\extd a=s^1\,[E_{12},a] + s^2\,[E_{21},a] 
\]
where $E_{ij}\in M_2(\C)$ has zero entries except for $1$ in the $ij$ position. We take $e_1,e_2$ to be the (central) dual basis of vector fields to $s^1,s^2$.
The $*$-operation is $s^1{}^*=-s^2$. 
 \hfill $\diamond$
\end{example}

\begin{example} \label{zncalcex}
The algebra $A=\C(\Z_n)$ of functions $f:\Z_n\to \C$ on the finite group $(\Z_n,+)$ with basis $\delta_i$ for $0\le i\le n-1$, which is the function $\delta_i(j)=\delta_{i,j}$. 
This has a calculus $\Omega^1_A$ with two non-central generators $e_{+1}$ and $e_{-1}$, where
\[ 
e_a.f=R_a(f)e_a\ ,\quad \extd f=e_{+1}(f-R_{-1}(f))+e_{-1}(f-R_{+1}(f))\ ,
\]
and $R_a(f)(i)=f(i+a)$ (mod $n$).
(This is a Hopf algebra with bicovariant calculus.) 
Let $\kappa_{\pm 1}$ be the dual basis of vector fields to $e_{\pm 1}$. The $*$-operation is $e_{\pm 1}^*=-e_{\mp 1}$. 
 \hfill $\diamond$
\end{example}

Bimodule connections were introduced in \cite{DVMass,DVMic,Mourad} and extensively used in \cite{Madore,FioMad}. However, here we use them in the more unusual context of mixed bimodules (different algebras on the left and right), and for that we refer to \cite{bbsheaf}. 
A $B$-$A$-bimodule $M$ is a left $B$-module and a right $A$-module, with the compatibility condition $(b.m).a=b.(m.a)$. 
This idea of bimodules strictly generalises the idea of the usual left or right modules for algebras over a field $\mathbb{K}$, as 
a  $B$-$\mathbb{K}$-bimodule is simply a left $B$-module and a $\mathbb{K}$-$A$-bimodule is simply a right $A$-module. 
We write ${}_B\CM_A$ as the category whose objects are $B$-$A$-bimodules, and whose morphisms are bimodule maps. Now suppose that $A$ and $B$ have differential calculi $\Omega_A$ and $\Omega_B$. 

\begin{definition} \label{tentat} A left $B$-$A$-bimodule connection on $M\in{}_B\CM_A$ means

\noindent{\rm (1)}\ A linear map $\nabla_M:M\to \Omega^1_B\tens_B M$ satisfying the left Liebniz rule
\[
\nabla_M(b.m)=\extd b\tens m+b.\nabla_M(m)\ ,\quad b\in B,\ m\in M.
\]
{\rm (2)}\  A $B$-$A$-bimodule map $\sigma_M:M\tens_A \Omega^1_A \to
\Omega^1_B\tens_B M$  such that 
\[ \nabla_M(m.a)=\nabla_M(m).a+\sigma_M(m\tens\extd a)\ ,\quad a\in A,\ m\in M.\]
\end{definition}

\medskip
An example of a left $C^\infty(\R^n)$-$C^\infty(\R^n)$-bimodule connection is a usual connection on the tangent space to $\R^n$. For $v=v^i\,\frac{\partial}{\partial x^i}$ we have
\begin{align} \label{rnconn}
\nabla(v^i\,\tfrac{\partial}{\partial x^i})=\extd x^k\tens( v^{i}{}_{,k}\,\tfrac{\partial}{\partial x^i}+v^i\,
\Gamma^j{}_{ki} \tfrac{\partial}{\partial x^j})\ ,\quad
\sigma(\tfrac{\partial}{\partial x^j}\tens\extd x^i)=\extd x^i \tens \tfrac{\partial}{\partial x^j}
\ .
\end{align}
Here the $\Gamma^j{}_{ki}$ are the usual Christoffel symbols, and note the common use of the subscript $,k$ for a partial derivative with respect to $x^k$. 

As originally noted in \cite{BDMS} for $A$-$A$-bimodules, but generalising to the current case, if we have $(Q,\nabla_Q,\sigma_Q)$ a left $C$-$B$-bimodule connection and $(M,\nabla_M,\sigma_M)$ a left $B$-$A$-bimodule connection, then we have a  left $C$-$A$-bimodule connection
on $Q\tens_B M$ by
\begin{align} \label{tpuor}
\nabla_{Q\tens M}(q\tens m)=\nabla_{Q}q\tens m+ (\sigma_{Q}\tens\id)(q\tens \nabla_Mm)\ ,\ \sigma_{Q\tens M}=(\sigma_Q\tens\id)(\id\tens\sigma_M)\ .
\end{align}
If $(M,\nabla_M,\sigma_M)$ and $(N,\nabla_N,\sigma_N)$ are left $B$-$A$-bimodule connections then
given a left module map $\theta:M\to N$ we define its derivative
\begin{align} \label{yerohd}
\doublenabla(\theta)=\nabla_N\,\theta-(\id\tens\theta)\,\nabla_M:M\to \Omega^1_B\otimes_B N\ .
\end{align}
The following result is an easy generalisation to the mixed context from \cite{BegMa:bia}.

\begin{proposition}  \label{bcuow}
For $\theta$ and $\doublenabla$ in (\ref{yerohd}), $\doublenabla(\theta)$ is a left module map. Further, supposing that $\theta$ is a bimodule map, then $\doublenabla(\theta)$ is a bimodule map if and only if $\sigma_N\circ(\theta\tens\id)=(\id\tens\theta)\circ\sigma_M$.
\end{proposition}
\begin{proof} 
Firstly, for $b\in B$ and $m\in M$,
\begin{align*}
\doublenabla(\theta)(bm)&=\nabla_N(b\,\theta(m))-(\id\tens\theta)\,\nabla_M(bm) \\
&= \extd b\tens \theta(m) + b.\nabla_N\theta(m)- \extd b\tens \theta(m) - b.(\id\tens\theta)\,\nabla_M(m)\ .
\end{align*}
Secondly, for $a\in A$,
\begin{align*}
\doublenabla(\theta)(ma)&=\nabla_N(\theta(m)a)-(\id\tens\theta)\,\nabla_M(ma) \\
&=\sigma_N(\theta(m)\tens\extd a) + \nabla_N(\theta(m)).a- (\id\tens\theta) \sigma_M(m\tens\extd a) - (\id\tens\theta)(\nabla_M(m)).a \\
&=\sigma_N(\theta(m)\tens\extd a)- (\id\tens\theta) \sigma_M(m\tens\extd a) + \doublenabla(\theta)(m).a\ .\qquad\square
\end{align*}
\end{proof}

 As we shall be concerned with positivity we shall deal exclusively with star algebras, and
 to avoid confusion we will be quite explicit about conjugate modules of these algebras.
 If $M$ is a $B$-$A$-bimodule then its conjugate $\overline{M}$ is an $A$-$B$-bimodule. Writing elements  $\overline{m}\in \overline{M}$ where $m\in M$ we have
 \[
 \overline{\lambda m+\mu n}= \lambda^*\,\overline{m}+ \mu^*\,\overline{n}\ ,\ 
 a.\overline{m}=\overline{m.a^*}\ ,\  \overline{m}.b=\overline{b^*.m}
 \]
 for all $m,n\in M$, $\lambda,\mu\in\C$, $a\in A$ and $b\in B$. For a $A$-$C$-bimodule $N$ we have the $C$-$B$-bimodule map 
 $\Upsilon:\overline{M\tens_A N}\to \overline{N} \tens_A \overline{M}$ defined by $\Upsilon(\overline{m\tens n})=\overline{ n}\tens
 \overline{m}$. 
 
 If $B$ is a $*$-algebra the $*$-operation is said to extend to $\Omega^1_B$ if
 there is a well defined antilinear map on $\Omega^1$ defined by $(a\,\extd b)^*=\extd b^*.a^*$. Given a left $B$-connection $\nabla_M$ on $M$ there is a right $B$-connection (i.e.\ satisfying a right version of the Liebniz rule in Definition~\ref{tentat})  $\tilde\nabla_{\overline{M}}$ on $\overline{M}$ given by $\tilde\nabla_{\overline{M}}(\overline{m})=\overline{n}\tens \xi^*$ where $\nabla(m)=\xi\tens n$ (see  \cite{BMriem}).

\section{Classical bimodule connections and geodesics} \label{classfun}
Consider manifolds $X$ of dimension $n$ and $Y$ of dimension $m$ with covariant derivatives on vector fields and forms given by
Christoffel symbols $\Gamma$ and $\Xi$ respectively. Suppose that there is a differentiable map $\gamma:Y\to X$. This induces an algebra map
$\tilde\gamma:C^\infty(X)\to C^\infty(Y)$ by $\tilde\gamma(f)=f\circ\gamma$. In particular, for coordinates $\{x^1,\dots,x^n\}$ of $X$ and $\{y^1,\dots,y^m\}$ of $Y$
we have $\tilde\gamma(x^i)=\gamma^i$ where we write $\gamma=(\gamma^1,\dots,\gamma^n)$. 

Given any right $C^\infty(Y)$-module $E$, we can define a right $C^\infty(X)$-module $E_{\tilde\gamma}$ by $e\,\ra\,f=e\,\tilde\gamma(f)$ for $e\in E$ and $f\in C^\infty(X)$. In particular we have a $C^\infty(Y)$-$C^\infty(X)$ bimodule $C^\infty(Y)_{\tilde\gamma}$, which is just $C^\infty(Y)$ as a left $C^\infty(Y)$ module. 

\begin{proposition} \label{propsigma}
For a differentiable map $\gamma:Y\to X$ we give we give $C^\infty(Y)_{\tilde\gamma}$ a left connection by
\[
\extd :C^\infty(Y)\to\Omega^1(Y)\cong \Omega^1(Y){\tens}_{ C^\infty(Y) } C^\infty(Y)\ .
\]
Then $\extd$ on $C^\infty(Y)_{\tilde\gamma}$ is a bimodule connection with
\[
\sigma(k\tens \extd f)=\extd(k\,\tilde\gamma(f)) - \extd(k)\,\tilde\gamma(f)=k\,\extd(\tilde\gamma(f))\in \Omega^1(Y)=\Omega^1(Y){\tens}_{ C^\infty(Y) } C^\infty(Y)
\]
for $k\in C^\infty(Y)_{\tilde\gamma}$ and $f\in C^\infty(X) $,
and in particular $\sigma(k\tens \extd x^i)=k\,\extd\gamma^i$. Also
\begin{align*}
\doublenabla(\sigma)(1\tens \extd x^i)&= \big(  \tfrac{\partial^2 \gamma^i}{\partial y^r\partial y^p}
- \tfrac{\partial \gamma^i}{\partial y^q}\, \Xi^q{}_{pr}
+(\Gamma^i{}_{jk}\circ\gamma)\,\tfrac{\partial \gamma^j}{\partial y^p} \,\tfrac{\partial \gamma^k}{\partial y^r} \big)\,\extd y^p\tens\extd y^r\ .
\end{align*}
\end{proposition}
\noindent\textbf{Proof:}\quad The formula for $\sigma$ is immediate from what is written. Next	
\begin{align*}
\doublenabla(\sigma)(1\tens \extd x^i)&= \nabla_{\Omega^1(Y)}\sigma(1\tens \extd x^i)
-(\id\tens\sigma)\nabla_{C^\infty(Y)\tens\Omega^1(X)}(1\tens \extd x^i) \\
&= \nabla_{\Omega^1(Y)}( \tfrac{\partial \gamma^i}{\partial y^r}\,\extd y^r)
+(\Gamma^i{}_{jk}\circ\gamma)(\id\tens\sigma)(\sigma\tens\id)(1\tens\extd x^j\tens\extd x^k) \ .\tag*{$\square$}
\end{align*}

\medskip In this classical case there is an alternative point of view.
 As $\gamma:Y\to X$ is differentiable it extends to $\tilde\gamma_*:\Omega^1(X)\to \Omega^1(Y)$ as a bimodule map, so we can define
\begin{align} \label{restricted}
\doublenabla_\gamma(\tilde\gamma_*)(\extd x^i)&=\big(\nabla_{\Omega^1(Y)}\circ\tilde\gamma_*-
(\tilde \gamma_*\tens\tilde \gamma_*)\circ \nabla_{\Omega^1(X)}\big)(\extd x^i)
\cr &= \nabla_{\Omega^1(Y)} (\tfrac{\partial \gamma^i}{\partial y^j} \,\extd y^j) -(\tilde \gamma_*\tens\tilde \gamma_*)(-\Gamma^i{}_{rs}\extd x^r\tens\extd x^s) \cr
&= \big(  \tfrac{\partial^2 \gamma^i}{\partial y^r\partial y^p}
- \tfrac{\partial \gamma^i}{\partial y^q}\, \Xi^q{}_{pr}
+(\Gamma^i{}_{jk}\circ\gamma)\,\tfrac{\partial \gamma^j}{\partial y^p} \,\tfrac{\partial \gamma^k}{\partial y^r} \big)\,\extd y^p\tens\extd y^r
\end{align}
so we see that $\doublenabla(\sigma)(1\tens \extd x^i)=\doublenabla_\gamma(\tilde\gamma_*)(\extd x^i)$. 

In the simplest case where $Y=\R$ with coordinate function $y$ and vanishing
Christoffel symbols $\Xi$ we have $\gamma:\R\to X$ and
\begin{align*}
\doublenabla(\sigma)(1\tens \extd x^i)&= \big(  \tfrac{\partial^2 \gamma^i}{\partial y^2}
+(\Gamma^i{}_{jk}\circ\gamma)\,\tfrac{\partial \gamma^j}{\partial y} \,\tfrac{\partial \gamma^k}{\partial y} \big)\,\extd y\tens\extd y
\end{align*}
and so $\doublenabla(\sigma)=0$ is the equation for $\gamma:\R\to X$ to be a geodesic. More generally for $Y$ an $m$-dimensional manifold  a geodesic in $Y$ with parameter $t$ is given by
$
\ddot y^i+\Xi^i{}_{jk}\,\dot y^j,\dot y^k=0
$. 
Then $\gamma(y(t))$ obeys
\begin{align*}
\ddot \gamma^i+&\Gamma^i{}_{jk}\,\dot \gamma^j\dot \gamma^k=\tfrac{\extd}{\extd t}(
\tfrac{\partial \gamma^i}{\partial y^r}\,\dot y^r)+\Gamma^i{}_{jk}\,
\tfrac{\partial \gamma^j}{\partial y^r}\,\dot y^r\,\tfrac{\partial \gamma^k}{\partial y^s}\,\dot y^s\cr
&=\tfrac{\partial \gamma^i}{\partial y^r}\,\ddot y^r+ \tfrac{\partial^2 \gamma^i}{\partial y^r\partial y^s}\,\dot y^r\,\dot y^s+\Gamma^i{}_{jk}\,
\tfrac{\partial \gamma^j}{\partial y^r}\,\dot y^r\,\tfrac{\partial \gamma^k}{\partial y^s}\,\dot y^s
=\big( \tfrac{\partial^2 \gamma^i}{\partial y^r\partial y^s} - \tfrac{\partial \gamma^i}{\partial y^p}\Xi^p{}_{rs}+\Gamma^i{}_{jk}\,
\tfrac{\partial \gamma^j}{\partial y^r}\,\tfrac{\partial \gamma^k}{\partial y^s}
\big)\dot y^r\dot y^s\ .
\end{align*}
As was pointed out by Sebastian Goette \cite{MO-sebgo}, this means that the
 condition $\doublenabla(\sigma)=0$ implies that 
every geodesic on $Y$ is mapped by $\gamma$ to a geodesic on $X$.

\medskip
For a Riemannian manifold the geodesics are the paths of (locally) minimal distance, and are given by the Levi Civita covariant derivative of the velocity vector along itself being zero. 
The reader should recall from standard theory (e.g.\ \cite{OneillRiem}) that a totally geodesic submanifold is one in which every geodesic in the big manifold starting at a point in the submanifold and with initial velocity along the tangent space to the submanifold remains in the submanifold for all time. Note that whereas the existence of geodesics is standard, the existence of
totally geodesic submanifolds of dimension strictly between 1 and the dimension of the manifold is a nontrivial condition on the manifold.  
It is important to note that we shall consider more general connections than just the Levi Civita one in this paper, with correspondingly more general `geodesics'.

The velocity of a geodesic is of paramount importance, and we note how it is encoded in the $\sigma$ notation. In the case $Y=\R$ with coordinate $t$, the velocity is simply $v^i=\frac{\extd\gamma^i(t)}{\extd t}$. The formula for $\sigma$ in Proposition~\ref{propsigma} reduces to
\begin{align} \label{sigcalc}
\sigma(1\tens \extd x^i)=\tfrac{\extd\gamma^i(t)}{\extd t}\,\extd t\tens 1.
\end{align}
Defining the velocity in isolation in more generality will have to wait until we have discussed the KSGNS construction for positive maps.

\section{Noncommutative paths and the KSGNS construction}

From the point of view of quantum mechanics, the Heisenberg uncertainty principle makes it likely that an idea of a geodesic as a single path will have to be replaced by a more uncertain or `probabilistic' idea, as we cannot precisely measure both position and velocity (momentum) at the same time, and the geodesic depends on both of these. 

In Section~\ref{classfun} we constructed the bimodule $C^\infty(Y)_{\tilde\gamma}$ from a function
$\gamma:Y\to X$ which induced an algebra map
$\tilde\gamma:C^\infty(X)\to C^\infty(Y)$. In noncommutative geometry it will prove 
impractical to follow this pattern, as in general there are simply not enough such algebra maps -- instead we shall look for completely positive maps. For $C^*$-algebras the completely positive maps are given by the KSGNS construction, and
for the general theory of Hilbert $C^*$-bimodules including a proper description of the KSGNS construction we refer to \cite{Lance}. We shall only use part of the KSGNS construction, and in particular shall not say that we consider all completely positive maps.

The algebras we consider are algebras of `differentiable functions' rather than $C^*$-algebras, but they are frequently dense $*$-subalgebras of $C^*$-algebras (or even local $C^*$-algebras as in \cite{BlackKth}). 
We write the conjugates in the inner products explicitly using the bar notation of Section~\ref{prelim} (to allow the use of connections), and swap sides of the conjugate when compared with \cite{Lance} (to continue using left connections).

\begin{definition}
For $A$ and $B$ which are dense $*$-subalgebras of some $C^*$ algebras
 and a $B$-$A$-bimodule $M$, a positive semi-inner product on $M$ is a bimodule map
$\<\,,\,\>:M\tens_A\overline{M}\to B$ which obeys $\<m,\overline{m'}\>^*=\<m',\overline{m}\>$ for all $m,m'\in M$ and where $\<m,\overline{m}\>$ is positive in $B$ for all $m\in M$. 
\end{definition}

I shall only briefly describe the KSGNS construction, as getting into details of the $C^*$-algebra construction is not required. Basically it says that completely positive maps $\psi:A\to B$ for $C^*$-algebras $A$ and $B$ are all given by $B$-$A$-bimodules $M$ with positive inner products
using the formula $\psi(a)=\<m.a,\overline{m}\>$ for some $m\in M$. One example is a Hilbert space, where $B=\C$, though this is usually written with the conjugate on the other side. A particular case would be the Schr\"odinger picture of quantum mechanics, where $A$ is the algebra of quantum observables and $M=L^2(\R^3)$.

\begin{example}
Now we explain what Proposition~\ref{propsigma} has to do with the KSGNS construction, where $B=C^\infty(Y)$, $A=C^\infty(X)$
and $M=C^\infty(Y)_{\tilde\gamma}$.  First define a positive inner product $\<\,,\>:M\tens \overline{M}\to 
C^\infty(Y)$ by $\<f,\overline{g}\>=f\,g^*$. Given the usual right $C^\infty(Y)$ action on the conjugate bimodule, $\overline{g}\ra k= \overline{k^*g}$, the inner product is a $C^\infty(Y)$-bimodule map. Using the usual left $C^\infty(X)$ action on the conjugate $\overline{M}$ corresponding to the right $C^\infty(X)$ action on $M$, we can check that we get a well defined map $\<\,,\>:M\tens_{C^\infty(X)} \overline{M}\to 
C^\infty(Y)$. To do this we use the following, restricting to the case of a real coordinate function $x^i$ on $X$ simply to avoid explicitly writing compositions
\begin{align*}
\<f\ra x^i,\overline{k}\>=\gamma^i\,\<f,\overline{k}\>,\quad \<f,x^i\la \overline{k}\>=\<f,\overline{k\ra x^i}\>=\gamma^i\,\<f,\overline{k}\>\ .
\qquad\qquad \diamond
\end{align*}
\end{example}

Now we consider a simple noncommutative example for a $*$-algebra $A$.

\begin{example} \label{exnofri}
Take a $C^\infty(\R)$-$A$-bimodule $M$ to be just $C^\infty(\R)\tens A$, with inner product
\begin{align} \label{yjkp}
\<f(t)\tens a,\overline{g(t)\tens b}\>_M=f(t)g(t)^*\,\<a, \overline{b}\>_A \in C^\infty(\R)
\end{align}
where $\<\,, \>_A$ is a fixed (time independent) inner product on $A$. 
\end{example}

For solving differential equations later we should really take  $\C^\infty(\R,A)$, the functions of time $t$ with values in $A$, rather than $C^\infty(\R)\tens A$, which is a proper subset if $A$ is infinite dimensional. However we ignore the technicalities required to define $\C^\infty(\R,A)$ and continue with $M=C^\infty(\R)\tens A$ as our examples are either finite dimensional or based on functions on manifolds. However, we shall write algebra or module valued functions of time at various places.

\section{Connections and the geodesic velocity equation}

We now consider the connections needed to generalise the classical Proposition~\ref{propsigma} to the noncommutative case of
Example~\ref{exnofri}.

\begin{proposition} \label{prido}
For a unital algebra $A$ with calculus $\Omega_A$ and $C^\infty(\R)$ with its usual calculus $\Omega(\R)$ we set  $M=C^\infty(\R)\tens A$ regarded as a $C^\infty(\R)$-$A$-bimodule. 
Then a general left bimodule connection on $M$ is of the form, for $c\in C^\infty(\R)\tens A$ and  $\xi\in\Omega^1_A$
\[
\nabla_M(c)=\extd t\tens \big(bc + \tfrac{\partial c}{\partial t} +K(\extd c)\big)\ ,\ \sigma_M(1\tens\xi)=\extd t\tens K(\xi)
\]
for some $b\in C^\infty(\R)\tens A$ and $K\in C^\infty(\R)\tens \cX_A$. [Note that explicitiy including time evaluation we have
$K(\eta)(t)=K(t)(\eta(t))$ for $\eta\in C^\infty(\R)\tens \Omega^1_A$.]
\end{proposition}
\begin{proof} The $C^\infty(\R)$-$A$-bimodule map $\sigma_M:M\tens_A \Omega^1_A\to \Omega^1(\R)\tens_{C^\infty(\R)} M$ is
uniquely specified by its value $\sigma_M(1\tens\xi) = \extd t\tens K(\xi)\in \Omega^1(\R)\tens_{C^\infty(\R)} M$ for $\xi\in\Omega^1_A$,  where $K:\Omega^1_A\to C^\infty(\R)\tens A$ is a right $A$-module map (i.e.\ a time dependent vector field), because $\sigma_M$ is a right $A$ module map. Next $\nabla_M(1)\in \Omega^1(\R)\tens_{C^\infty(\R)} M$ must be  $\extd t\tens b$ for some $b\in C^\infty(\R)\tens A$.
Now put $c(t)=f(t)\,a$ for $f\in\C^\infty(\R)$ and $a\in A$ and calculate
\begin{align*}
\nabla_M(fa)&=\tfrac{\partial f}{\partial t}\,\extd t\tens a+f\,\nabla_M(a)
=\tfrac{\partial f}{\partial t}\,\extd t\tens a+f\,\nabla_M(1)a+f\,(\nabla_M(1.a)-\nabla_M(1)a)
\\
&= \tfrac{\partial f}{\partial t}\,\extd t\tens a+f\,\nabla_M(1)a+f\,\sigma_M(1\tens\extd a)\ .\qquad\qquad\square
\end{align*}
\end{proof}

\begin{proposition} \label{prido99}
Take $A$ and $(M,\nabla_M,\sigma_M)$ as in Proposition~\ref{prido}. Also take the trivial connection on $\Omega^1(\R)$ (i.e.\ $\nabla_{\Omega^1(\R)}(f(t)\,\extd t)=\extd t\tens \tfrac{\partial f}{\partial t}\,\extd t$ and $\sigma_{\Omega^1(\R)}(\extd t\tens \extd t)=\extd t\tens \extd t$) and a left bimodule connection $\nabla_{\Omega^1_A}$ on the $A$-bimodule $\Omega^1_A$
with invertible $\sigma_{\Omega^1_A}:\Omega^1_A\tens_A\Omega^1_A\to \Omega^1_A\tens_A\Omega^1_A$. Then $\doublenabla(\sigma_M)=0$ if and only if both $K(K\tens\id)\sigma_{\Omega^1_A}=K(K\tens\id)$ and  for all $\xi\in\Omega^1_A$ (constant in time)
\begin{align}\label{gedef}
 \tfrac{\partial K(\xi)}{\partial t}  = K(b\xi) - bK(\xi) +K(K\tens\id)\sigma_{\Omega^1_A}^{-1}\nabla_{\Omega^1_A}(\xi) - K(\extd K(\xi))\ .
\end{align}
Further, if $K$ satisfies (\ref{gedef}) for $\xi\in\Omega^1_A$ then it also satisfies it for $\xi\,a\in\Omega^1_A$ for all $a\in A$, so it is only necessary to verify (\ref{gedef})
for a collection of right generators of $\Omega^1_A$.
\end{proposition}
\begin{proof} Writing $\nabla_{\Omega^1_A}(\xi)=\eta\tens\kappa$, we find
\begin{align} \label{yoeya}
\doublenabla(\sigma_M)&(1\tens\xi)=
\nabla_{\Omega^1(\R)\tens M}\big(\sigma_M(1\tens\xi)\big)
-(\id\tens\sigma_M)\nabla_{M\tens {\Omega^1_A}}(1\tens\xi) \cr
&= \nabla_{\Omega^1(\R)\tens M}(\extd t\tens K(\xi)) -  (\id\tens\sigma_M)\big( \extd t\tens b\tens\xi+\sigma_M(1\tens\eta)\tens\kappa\big)   \cr
&= \sigma_{\Omega^1(\R)}(\extd t\tens\extd t)\tens \big(
bK(\xi) + \tfrac{\partial K(\xi)}{\partial t} +K(\extd K(\xi))\big) \cr
&\qquad - \extd t\tens\extd t\tens \big(K(b\xi)+K(K(\eta)\kappa)\big)
\end{align}
so 
 we get $\doublenabla(\sigma_M)=0$ reducing to
\begin{align} \label{procw}
bK(\xi) + \tfrac{\partial K(\xi)}{\partial t} +K(\extd K(\xi)) = K(b\xi)+K(K\tens\id)\nabla_{\Omega^1_A}(\xi)\ .
\end{align}
If $\doublenabla(\sigma_M)=0$ from Proposition~\ref{bcuow} we also have $\sigma_{{\Omega^1(\R)}\tens M}(\sigma_M\tens\id)=(\id\tens\sigma_M)\sigma_{M\tens \Omega^1_A}$, which is just the braid relation 
\begin{align} \label{outfd}
(\sigma_{\Omega^1(\R)}\tens\id) (\id\tens\sigma_M)  (\sigma_M\tens\id)    =   (\id\tens\sigma_M)    (\sigma_M\tens\id)  (\id\tens\sigma_{\Omega^1_A})
\end{align}
and this gives the equation $K(K\tens\id)\sigma_{\Omega^1_A}=K(K\tens\id)$. Using this and the invertibility of $\sigma_{\Omega^1_A}$ in (\ref{procw}) gives
(\ref{gedef}). To check the right multiplication property for (\ref{gedef}) we look at
\begin{align*}
K(&K\tens\id)\sigma_{\Omega^1_A}^{-1}\nabla_{\Omega^1_A}(\xi\,a) - K(\extd K(\xi\,a))\\  &= K(K\tens\id)\sigma_{\Omega^1_A}^{-1}\nabla_{\Omega^1_A}(\xi)a
+K(K\tens\id)(\xi\tens\extd a)  - K(\extd( K(\xi)a)) \\
 &= K(K\tens\id)\sigma_{\Omega^1_A}^{-1}\nabla_{\Omega^1_A}(\xi)a
+K(K\tens\id)(\xi\tens\extd a)  - K(\extd K(\xi))a-K( K(\xi)\,\extd a)\ .\qquad \square
\end{align*}
\end{proof}

\medskip
As (\ref{gedef}) is in the form of a time evolution for $K$ as a function of $t$, it is natural to specify $K$ at time zero and then try to solve (\ref{gedef}) for some time interval including $t=0$. We must then assume that $K(K\tens\id)(\sigma_{\Omega^1_A}-\id)=0$ at time zero. But does this remain true under the time evolution given by (\ref{gedef})?

\begin{cor} \label{ldtrax} If we set $G=K(K\tens\id)(\sigma_{\Omega^1_A}-\id):\Omega^1_A\tens_A \Omega^1_A\to A$ then
\begin{align*}
& \tfrac{\partial G}{\partial t}  =GL_b-L_bG+ K(K\tens\id)(K\tens\id\tens\id)\doublenabla(\sigma_{\Omega^1_A})
+ G(K\tens\id\tens\id)\nabla_{\Omega^1_A\tens \Omega^1_A}  - K\extd G
\cr
&\qquad - K(G\sigma_{\Omega^1_A}^{-1}\tens\id)\nabla_{\Omega^1_A\tens \Omega^1_A}(\sigma_{\Omega^1_A}-\id) - G\sigma_{\Omega^1_A}^{-1} \nabla_{\Omega^1_A}(K\tens\id)  (\sigma_{\Omega^1_A}-\id)
\end{align*}
where $L_b$ is the left multiply by $b$ operation, $\nabla_{\Omega^1_A\tens \Omega^1_A}$ is the tensor product connection on 
$\Omega^1_A\tens_A\Omega^1_A$ and
$\doublenabla(\sigma_{\Omega^1_A})=\nabla_{\Omega^1_A\tens \Omega^1_A}\sigma_{\Omega^1_A}-(\id\tens \sigma_{\Omega^1_A})\nabla_{\Omega^1_A\tens \Omega^1_A}$. 
\end{cor}
\begin{proof} 
From (\ref{gedef}), for $\xi,\mu\in\Omega^1_A$,
\begin{align*}
& \tfrac{\partial K(K\tens \id)(\xi\tens\mu)}{\partial t}  = K\big(\big(K(b\xi) - bK(\xi) +K(K\tens\id)\sigma_{\Omega^1_A}^{-1}\nabla_{\Omega^1_A} (\xi) - K(\extd K(\xi))\big)\mu\big)\\
&\qquad + K(bK(\xi)\mu) - bK(K(\xi)\mu) +K(K\tens\id)\sigma_{\Omega^1_A}^{-1}\nabla_{\Omega^1_A} (K(\xi)\mu) - K(\extd K(K(\xi)\mu)) \\
&\quad =  K\big(K(b\xi)\mu\big) +  K(K\tens\id)(K\tens\id\tens\id)\big(\sigma_{\Omega^1_A}^{-1}\nabla_{\Omega^1_A} (\xi)\tens \mu\big) 
- K(K\tens\id)\big( \extd K(\xi)\tens \mu\big)
\\
&\qquad  - bK(K(\xi)\mu) +K(K\tens\id)\sigma_{\Omega^1_A}^{-1}\nabla_{\Omega^1_A} (K(\xi)\mu) - K(\extd K(K(\xi)\mu)) \\
&\quad =  K\big(K(b\xi)\mu\big)   - bK(K(\xi)\mu) +  K(K\tens\id)(K\tens\id\tens\id)\big(\sigma_{\Omega^1_A}^{-1}\nabla_{\Omega^1_A} (\xi)\tens \mu
+\xi\tens \sigma_{\Omega^1_A}^{-1}\nabla_{\Omega^1_A} (\mu)\big) 
\\
&\qquad + K(K\tens\id)(\sigma_{\Omega^1_A}^{-1}-\id)\big( \extd K(\xi)\tens \mu\big)   - K\extd (K(K\tens\id)(\xi\tens\mu)) \ .
\end{align*}
If we put $S=K(K\tens\id)$ then we find
\begin{align*}
& \tfrac{\partial S(\xi\tens\mu)}{\partial t}  =S(b\xi\tens\mu)-bS(\xi\tens\mu)+ K(S\sigma_{\Omega^1_A}^{-1}\tens\id)\nabla_{\Omega^1_A\tens \Omega^1_A}(\xi\tens\mu)\\
&\qquad + S(\sigma_{\Omega^1_A}^{-1}-\id)\big( \extd K(\xi)\tens \mu+K(\xi)\nabla_{\Omega^1_A} (\mu)\big)   - K\extd S(\xi\tens\mu) \\
&\quad =S(b\xi\tens\mu)-bS(\xi\tens\mu)+ K(S\tens\id)\nabla_{\Omega^1_A\tens \Omega^1_A}(\xi\tens\mu) 
\\
&\qquad + K(S(\sigma_{\Omega^1_A}^{-1}-\id)\tens\id)\nabla_{\Omega^1_A\tens \Omega^1_A}(\xi\tens\mu) + S(\sigma_{\Omega^1_A}^{-1}-\id) \nabla_{\Omega^1_A} (K\tens\id)(\xi\tens\mu)   - K\extd S(\xi\tens\mu) 
\end{align*}
or in terms of operators, 
\begin{align} \label{upreus}
& \tfrac{\partial S}{\partial t}  =SL_b-L_bS+ S(K\tens\id\tens\id)\nabla_{\Omega^1_A\tens \Omega^1_A}  - K\extd S
\cr
&\qquad - K(S(\sigma_{\Omega^1_A}-\id)\sigma_{\Omega^1_A}^{-1}\tens\id)\nabla_{\Omega^1_A\tens \Omega^1_A} - S(\sigma_{\Omega^1_A}-\id)\sigma_{\Omega^1_A}^{-1} \nabla_{\Omega^1_A} (K\tens\id)  
\end{align}
and then compose this with $(\sigma_{\Omega^1_A}-\id)$.\hfill $\square$
\end{proof}

\medskip
In the cases where $\doublenabla(\sigma_{\Omega^1_A})=0$ (such as classical geometry) or where we can otherwise ensure the vanishing of 
$K(K\tens\id)(K\tens\id\tens\id)\doublenabla(\sigma_{\Omega^1_A})$ we see that $G=0$ is a solution of equation (\ref{gedef}) on the interval where $K$ is defined. Thus, if we have uniqueness of solution of the equation we would have $K(K\tens\id)(\sigma_{\Omega^1_A}-\id)=0$. 
Now we need to justify why the equation $\doublenabla(\sigma_M)=0$ in Proposition~\ref{prido99} in the classical case has anything to do with geodesics.

\begin{example} \label{classcomp}
We consider the classical case with algebra $A=C^\infty(X)$ for a $n$-dimensional manifold $X$. 
The equation  (\ref{gedef}) for the time dependent vector field $K$
 on $X$ becomes
\begin{align} \label{yodfy}
 \tfrac{\partial K^i}{\partial t}  + K^s\, K^i{}_{,s}  &+ K^k\,K^j \,\Gamma^i{}_{jk}
 =0
\end{align}
where $\Gamma^i{}_{jk}$ are the Christoffel symbols for the connection on $X$. 
Now, suppose that we start a point at $x(0)\in X$ for $t=0$ and move it according to the vector field $\tfrac{\extd x}{\extd t}=K(x)$. As the point moves, the
`convective derivative' from fluid mechanics gives  $\tfrac{\extd K^i(x)}{\extd t}= \tfrac{\partial K^i}{\partial t} + K^s\, K^i{}_{,s} $ and so 
(\ref{yodfy}) becomes $\tfrac{\extd K^i(x)}{\extd t}+ K^k\,K^j \,\Gamma^i{}_{jk}=0$, which is the usual equation for the velocity being parallel transported. Thus the vector field approach here is actually giving the velocity field for particles obeying geodesic motion starting at arbitrary points. Note that $\sigma_A$ is just the flip map and the extra condition $K(K\tens\id)(\sigma_A-\id)=0$ is automatically satisfied. \hfill$\diamond$
\end{example}

\section{Connections and the KSGNS construction} \label{plurg}
Recall that for a $\C^\infty(\R)$-$A$-bimodule $M$ with a positive inner product $\<\,,\>:M\tens_A \overline{M}\to \C^\infty(\R)$, given a $m\in M$ we have a positive map $\psi:A\to C^\infty(\R)$ given by $\psi(a)=\<m.a,\overline{m}\>$. But what is $m$? In Proposition~\ref{propsigma} where $M=C^\infty(\R)_{\tilde\gamma}$ we had $m(t)=1$, so given the usual $t$-derivative on $M$ we had
$\nabla_M(m)=0$, and it will turn out that this is a reasonable condition to assume in the noncommutative case.

\begin{example} \label{classar} We continue from Example~\ref{classcomp} and check whether
the equation $\nabla_M(m)=0$ corresponds to the classical situation. 
Take a positive function on $C^\infty(\R^3)$ given by
\[
\phi(t)(f) =\int_{\R^3} |m(t)|^2\,f\, \extd^3 x
\]
for the usual Lebesgue measure on $\R^3$ and a time dependent rapidly decreasing function (or rather density) $m(t)\in L^2(\R^3,\C)$. 
We use the connection $\nabla_M(m)=m\,b+K(\extd m)+\frac{\partial m}{\partial t}$  from 
Proposition~\ref{prido}.
We suppose that $K$ is a real vector field which is constant in time, and that $b=0$. Then $K$ gives an action of 
 $(\R,+)$ on the manifold $X$ by the flow $F:\R\times X\to X$ given by $\frac{\partial F(t,x)}{\partial t} =K(F(t,x))$ (i.e.\ a tangent vector at $F(t,x)$). 
Now the equation $\nabla_M(m)=0$ is solved by setting $m(t)(x)=m(0)(F(-t,x))$. If at time $0$ we have $m$ concentrated at the point $x_0$ then at time $t$ it is concentrated at $x_t$ where $F(-t,x_t)=x_0$, i.e.\ $x_t=F(t,x_0)$. 
 \hfill $\diamond$
\end{example}

Now we have several problems:

\smallskip

1) The flow in Example~\ref{classar} does not preserve the normalisation of the positive function (i.e.\ its value on $1\in A$ is time dependent) in general for a vector field $K$ with non-zero divergence. 

2) The $b$  from Proposition~\ref{prido99} doesn't appear in Example~\ref{classcomp}, as $b$ cancels from (\ref{gedef}) as a classical vector field is a bimodule map. However, in general the $b$ terms will contribute to the velocity equation. But there is no equation for $\frac{\extd b}{\extd t}$, so how to find $b$? 

3) For a classical manifold it is obvious what a real vector field is, but in noncommutative geometry it is not at all obvious in general. The reader should recall that we cannot immediately define a real vector field as the `obvious definition' $K=K^*$ as $K^*(\xi)=(K(\xi^*))^*$ swaps left and right vector fields.

\smallskip To consider these further, we need a definition:

\begin{definition} \label{quaty} For a $B$-valued inner product $\<\,,\>$ on a $B$-$A$-bimodule $M$, we say that a connection
$\nabla_M:M\to \Omega^1_B\tens_B M$ preserves the inner product if for all $n,m\in M$
\begin{align*} 
\extd \<n,\overline{m}\> =
(\id\tens\<\,,\>)(\nabla_M(n)\tens \overline{m}) + (\<\,,\>\tens\id)(n\tens \tilde\nabla_{\overline{M}}(
\overline{m})) \in\Omega^1_B
\end{align*}
where $\tilde\nabla_{\overline{M}}$ is the right connection  from Section~\ref{prelim}. 
\end{definition}

Now we apply this definition to the $\R^3$ case of Example~\ref{classar}.

\begin{example} \label{classar5} 
Begin with the inner product on time dependent elements of $ L^2(\R^3,\C)$ 
\[
\<k,\overline{c}\>=\int kc^*\,\extd^3 x\ .
\]
From Proposition~\ref{prido} the equation for $\nabla_M$ preserving the metric becomes
\begin{align*}
\frac{\extd}{\extd t}\int kc^*\,\extd^3 x &= \int\Big( (bk+\dot k+K^i\tfrac{\partial k}{\partial x^i})c^*+k(b^*c^*+\dot c^*+K^i{}^*\tfrac{\partial c^*}{\partial x^i})\Big)\,\extd^3 x
\end{align*}
and by using integration by parts we get, for all $c,k\in L^2(\R^3,\C)$ 
\begin{align*}
 \int kc^*(b+b^*-\tfrac{\partial K^i}{\partial x^i})\,\extd^3 x +  \int k \tfrac{\partial c^*}{\partial x^i} (K^i{}^*-K^i)\,\extd^3 x =0
 \ .
\end{align*}
This requires that $K^i{}^*=K^i$ (i.e.\ $K$ is a real vector field) and that 
 the real part of $b(t)$ is half the usual divergence of the vector field $K(t)$. The imaginary part of $b(t)$ is arbitrary -- this could be thought of as a gauge choice. \hfill$\diamond$
\end{example} 

Having considered a classical special case we now go to a more general case.

\begin{proposition}  \label{presip}
The connection on $M=C^\infty(\R)\tens A$ in Proposition~\ref{prido} preserves the inner product on $M$ if
for all $a\in A$ and $\xi\in\Omega^1_A$
\[
\<\big(ba+K(\extd a)+ab^*
\big),\overline{1}\>=0 =\<K(\xi^*)-K(\xi)^*,\overline{1}\>\ .
\]
\end{proposition}
\noindent\textbf{Proof:}\quad The condition for preservation is, for $a,c\in A$,
\begin{eqnarray*}
0 &=& \extd t.\<   ba+K(\extd a) ,  \overline{c}\> + \<a,\overline{bc+K(\extd c)}\>.\extd t \cr
&=& \extd t.\<\big(bac^*+K(\extd a)c^*+ac^*b^*+aK(\extd c)^*
\big),\overline{1}\> \cr
&=& \extd t.\<\big(bac^*+K(\extd a.c^*)+ac^*b^*+K(\extd c.a^*)^*
\big),\overline{1}\> 
\end{eqnarray*}
and putting $c=1$ gives the first displayed equation. Using this with $ac^*$ instead of $a$ in the condition for preservation gives
\begin{eqnarray*}
0 = \<\big(K(\extd a.c^*)-K(\extd(ac^*))+K(\extd c.a^*)^*
\big),\overline{1}\> = \<\big(K(\extd c.a^*)^*-K( a.\extd c^*)
\big),\overline{1}\> \ .\quad\square
\end{eqnarray*}

\medskip We call the first of the displayed equations in Proposition~\ref{presip} the \textit{divergence condition} for $b$ and the second the \textit{reality condition} for $K$. This answers questions (2) and (3), for question (1) we have the following:

\begin{proposition} \label{orbg}
If $\nabla_M$ preserves the inner product as in Definition~\ref{quaty} and 
$\nabla_M(m)=0$ then 
the positive map $\phi(a)=\<ma,\overline{m}\>$ satisfies 
\[
\extd t.\tfrac{\extd}{\extd t} \phi(a)=(\id\tens\<\,,\>)(\sigma_M\tens\id)(m\tens \extd a\tens\overline{m}) 
\ .
\]
In particular
 $\frac{\extd}{\extd t}\phi(1)=0$, so if we begin at $t=0$ with a state on $A$ (normalised to be 1 at $1\in A$) then we have a state for all time.
\end{proposition}
\begin{proof} From Definition~\ref{quaty} in the case where $B=C^\infty(\R)$,
\begin{align*} 
\extd t.\tfrac{\extd}{\extd t} \<m.a,\overline{m}\> =
(\id\tens\<\,,\>)(\nabla_M(m.a)\tens \overline{m}) + (\<\,,\>\tens\id)(m.a \tens \tilde\nabla_{\overline{M}}(
\overline{m}))
\end{align*}
and then use the definition of $\sigma_M$ in 
Definition~\ref{tentat}. \hfill $\square$
\end{proof} 

\medskip 
Classically the velocity $V$ of a path at a particular time is a vector at a single point. More generally, for a quantum path we expect to have to average over points to get a numerical value. Thus we consider $V$ evaluated al $\xi\in\Omega^1_A$ to be
$\<m.\xi,\overline{m}\>$, where we use $\sigma_M$ to rearrange this formula to make sense, giving
\[
V(\xi)=(\id\tens\<\,,\>)(\sigma_M\tens\id)(m\tens \xi\tens\overline{m}) \in C^\infty(\R)\ .
\]
This formula can be justified by 
Proposition~\ref{orbg}, whose result can then be written as $\tfrac{\extd}{\extd t} \phi(a)=V(\extd a)$. 
The obvious next question is if reality of vector fields is preserved by the time evolution given by (\ref{gedef}).
As in Corollary \ref{ldtrax} the next result is phrased to avoid any assumptions of uniqueness of solutions.

\begin{proposition}Suppose that
\[
\sigma_{\Omega^1_A}^{-1}\nabla_{\Omega^1_A}(\xi^*)=\dag\,\nabla_{\Omega^1_A}(\xi)
\]
where $\dag$ is flip on tensor factors composed with $*\tens *$, which is the statement that the connection on $\Omega^1_A$ preserves the $*$-operation \cite{BMriem}. Also
suppose that $b$ and $K$ satisfy the equations given in Proposition~\ref{prido99} as a consequence of $\doublenabla(\sigma_M)=0$, and that at a given time $b$ and $K$ satisfy the conditions in the statement of Proposition~\ref{presip}. Then at the given time
\[
\tfrac{\extd}{\extd t}\<K(\xi^*)-K(\xi)^*,\overline{1}\>=0\ .
\]
\end{proposition}
\begin{proof} First note that for $c\in C^\infty(\R)\tens A$ we have
$
\<c^*,\overline{1}\> = \<1,\overline{c^*}\>^*=\<c,\overline{1}\>^*
$.
From Proposition~\ref{presip},
\begin{align*}
&\<K(b\xi^*) - bK(\xi^*)  - K(\extd K(\xi^*)),\overline{1}\> = 
\<K(b\xi^*) +K(\xi^*)b^*  ,\overline{1}\> \cr
&\quad = \<K((\xi b^*)^*) +K((b\xi)^*)  ,\overline{1}\> = \<\big(K(\xi b^*) +K(b\xi) \big)^* ,\overline{1}\> \cr
& = \<K(\xi) b^* +K(b\xi)  ,\overline{1}\>^*
= \< K(b\xi) - bK(\xi)  - K(\extd K(\xi))  ,\overline{1}\>^* \ .
\end{align*}
In remains to deal with the $\nabla_{\Omega^1_A}$ term in (\ref{gedef}). Setting $\nabla_{\Omega^1_A}(\xi)=\eta\tens\kappa$ we get
\begin{align*}
&\<K(K\tens\id)\sigma_{\Omega^1_A}^{-1}\nabla_{\Omega^1_A}(\xi^*) ,\overline{1}\> = \< K(K(\kappa^*)\eta^*),\overline{1}\>
= \< K(\eta K(\kappa^*)^*)^*,\overline{1}\> 
 = \< K(\kappa^*)K(\eta )^*,\overline{1}\> \cr
 &\quad = \< K(\eta )K(\kappa^*)^*,\overline{1}\>^*=\< K(\kappa^*K(\eta )^*)^*,\overline{1}\>^*  = \< K(K(\eta )\kappa),\overline{1}\>^* \cr
&\quad =
\<  K(K\tens\id)\nabla_{\Omega^1_A}(\xi),\overline{1}\>^* = \<  K(K\tens\id)\sigma_{\Omega^1_A}^{-1}\nabla_{\Omega^1_A}(\xi),\overline{1}\>^*\ .
\qquad\qquad \square
\end{align*}
\end{proof}

\section{A $C(\Z_3)$ example} \label{z3exsect}
 Take the algebra $A=\C(\Z_n)$ of functions $f:\Z_n\to \C$ with calculus as in Example~\ref{zncalcex}.
 For the inner product we use
\[
\<a,\overline{c}\>=\int ac^*
\]
where $\int a$ is summation of $a(i)$ over $i\in\Z_n$. For the time dependent vector field $K$ we set
$K(e_{\pm 1})=K_\pm\in A$.
From Proposition~\ref{presip}
 preserving the inner product gives the reality condition, for all $a\in A$
\[
0=\int\big(K(a\,e_{\pm 1})^*-K((a\,e_{\pm 1})^*)\big)
=\int \big( K_\pm^* R_{\mp1}(a^*)+K_{\mp}a^*\big)
=\int \big( R_{\pm1}(K_\pm^*) +K_{\mp}\big)a^*
\]
which gives $K_-=-R_{1}(K_+^*)$. The other condition from  Proposition~\ref{presip} is the divergence condition, and for this we calculate
\[
K(\extd c)=K\big(e_{+1}(c-R_{-1}(c))+e_{-1}(c-R_{+1}(c))\big)
=K_+(c-R_{-1}(c))+ K_-(c-R_{+1}(c))\ ,
\]
so the divergence condition is
\begin{align*}
0&= \int\big(   bc+cb^* +K_+(c-R_{-1}(c))+ K_-(c-R_{+1}(c))  \big) \cr
&= \int c \big(  b+b^* + K_+   -  R_{+1}(K_+)+ K_-  - R_{-1}(K_-)\big)
\end{align*}
so we require
\[
0= b+b^* + K_+   -  R_{+1}(K_+)+ K_-  - R_{-1}(K_-)\ ,
\]
and by the reality condition $K_-=-R_{+1}(K_+^*)$ we can set
\[
b+K_++K_-=\tfrac12(K_+   +  R_{+1}(K_+) + K_-  + R_{-1}(K_-))\ .
\]
 
 We set $\nabla_A(e_{\pm 1})=0$, which gives  $\sigma_A(e_a\tens e_b)=e_b\tens e_a$. 
 Then $\doublenabla(\sigma_M)=0$ if and only if both $K(K\tens\id)\sigma_{\Omega^1_A}=K(K\tens\id)$ and (\ref{gedef}) holds, which in this case is
\begin{align}\label{gedeft}
 \tfrac{\partial K(e_{\pm 1})}{\partial t}  &= K(be_{\pm 1}) - b\, K(e_{\pm 1})  - K(\extd K(e_{\pm 1})) = K(e_{\pm 1}R_{\mp1}(b)) - b\, K_{\pm} -  K(\extd K_\pm) \cr
  &= K_\pm R_{\mp1}(b) - b\,  K_\pm   -K\big(e_{+1}(K_\pm-R_{-1}(K_\pm))+e_{-1}(K_\pm-R_{+1}(K_\pm))  \big) \cr
  &= K_\pm R_{\mp1}(b) - b\,  K_\pm  - K_+(K_\pm-R_{-1}(K_\pm)) - K_-(K_\pm-R_{+1}(K_\pm))\ .
\end{align}
(Recall that we only have to solve (\ref{gedef}) on the generators.)
As
\[
K(K\tens\id)(e_{+1}\tens e_{-1}) = K(K_+.e_{-1})=K(e_{-1}.R_{+1}(K_+))=K_-\,R_{+1}(K_+)
\]
and similarly for $e_{-1}\tens e_{+1}$,
 the $K(K\tens\id)(\sigma_A-\id)=0$ condition corresponds to $K_-\,R_{+1}(K_+)=K_+\,R_{-1}(K_-)$, i.e.\
 $K_-(i)\,K_+(i+1)=K_-(i-1)\,K_+(i)$. Then, using this
 \begin{align}\label{gedeft7}
 \tfrac{\partial K_+ }{\partial t}  
  &= K_+\big( R_{-1}(b) - b  - K_+  - K_-\big)
  + K_+R_{-1}(K_+) 
  + K_-R_{+1}(K_+) \cr 
    &= K_+\big( R_{-1}(b) - b  - K_+  - K_-\big)
  + K_+R_{-1}(K_+) 
  + K_+\,R_{-1}(K_-)\cr 
     &= K_+\big( R_{-1}(b + K_+  + K_-) - (b  + K_+  + K_-)\big)  \ ,   \cr
 \tfrac{\partial K_- }{\partial t}  
  &= K_-\big( R_{+1}(b) - b  - K_+  - K_-\big)
  + K_+R_{-1}(K_-) 
  + K_-R_{+1}(K_-) \cr
      &= K_-\big( R_{+1}(b+ K_+  + K_-) - (b  + K_+  + K_-) \big)\ .
\end{align}
Now we calculate
\begin{align*}
R_{+1}(b+K_++K_-)-(b+K_++K_-) 
&=
\tfrac12(  R_{+2}(K_+) + R_{+1}(K_-)  )
- \tfrac12(K_+     + R_{-1}(K_-)) 
\cr
R_{-1}(b+K_++K_-)-(b+K_++K_-) 
&=
\tfrac12(R_{-1}(K_+)     + R_{-2}(K_-))
- \tfrac12(  R_{+1}(K_+) + K_-  )
\end{align*}          
and these values can be substituted in (\ref{gedeft7}) to give differential equations for $K_\pm$. 
Using the reality condition and 
$K(K\tens\id)(\sigma_A-\id)=0$ we deduce that $|R_{+1}(K_+)|^2=|K_-|^2=|K_+|^2$, so
 $|K_\pm|$ is constant. 
 Using the connection in Proposition~\ref{prido} $\nabla_M(m)=0$ becomes
 \begin{align} \label{protr}
 \tfrac{\extd m}{\extd t} = -mb-K_+(m-R_-(m))-K_-(m-R_+(m))\ .
\end{align}
For the case $n=3$ we solve we solve these equations numerically with the
initial conditions for the vector field $K$ and $m\in M$ at $t=0$
\begin{align}\label{gedef009ulp} 
&K_-(1)= 1,\, K_-(2)= \mathrm{e}^{2\mathrm{i}},\, K_-(0)= \mathrm{e}^{3\mathrm{i}},\, K_+(1)= -\mathrm{e}^{-3\mathrm{i}},\, K_+(2)= -1, \,
K_+(0)= -\mathrm{e}^{-2\mathrm{i}}\cr
& m(0)=\tfrac{1}{\sqrt 2}\ ,\  m(1)=0\ ,\  m(2)=\tfrac{1}{\sqrt 2}\ .
\end{align}

 In Figure~\ref{highden35} (a) the three graphs represent the state evolving in time, with the lower graph being $\phi(\delta_0)$, the middle being
$\phi(\delta_0+\delta_1)$ and the upper $\phi(\delta_0+\delta_1+\delta_2)$. Then (b) shows the deviation from $K$ being real by plotting
the absolute values of $K_-(i)+K_+(i+1)^*$ for $i\in\{0,1,2\}$. Finally (c) plots $\phi(\delta_0+\delta_1+\delta_2)-1$,
 the deviation of the numerical solution from preserving the normalisation of the state. Neither of these properties are explicitly imposed on the numerical solution except at time $t=0$.

 \begin{figure}
    \centering
    \subfloat[the state $\phi(\delta_i)=|m(i)|^2$]{{\includegraphics[scale=0.33]{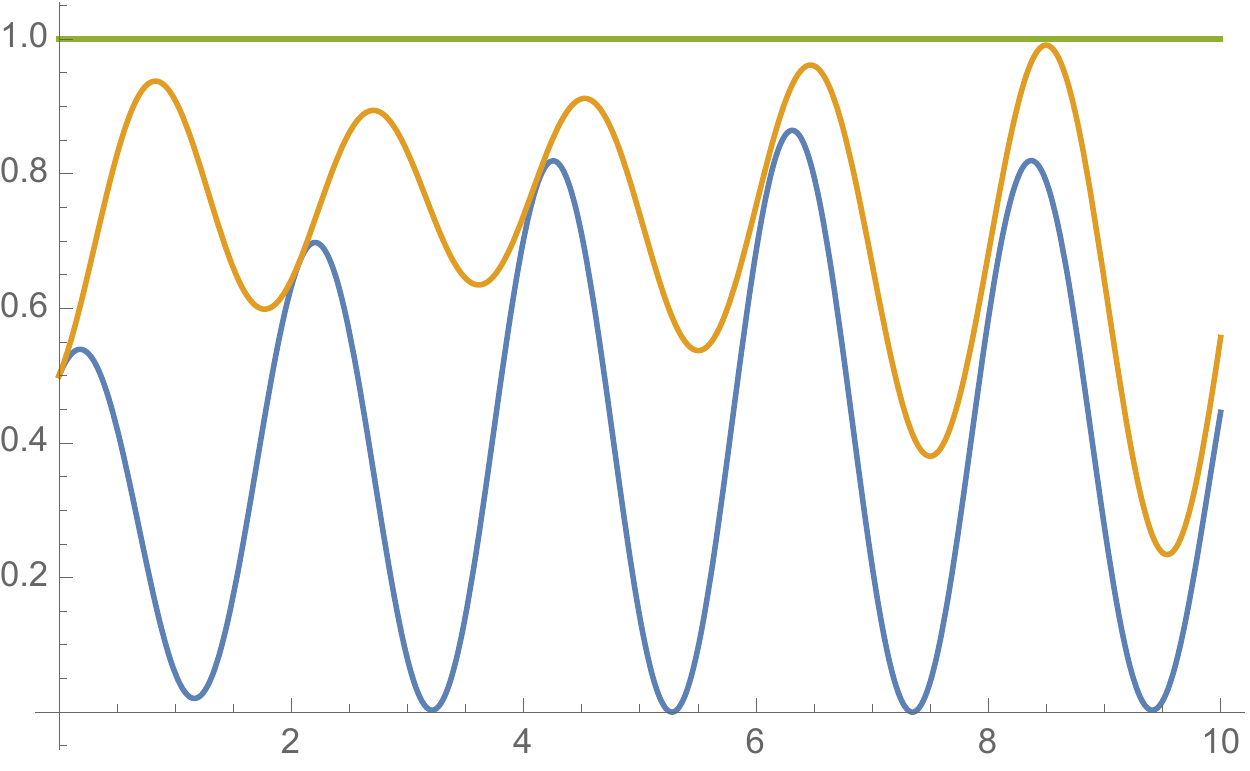}   } }%
    \,
    \subfloat[reality check]{{\includegraphics[scale=0.33]{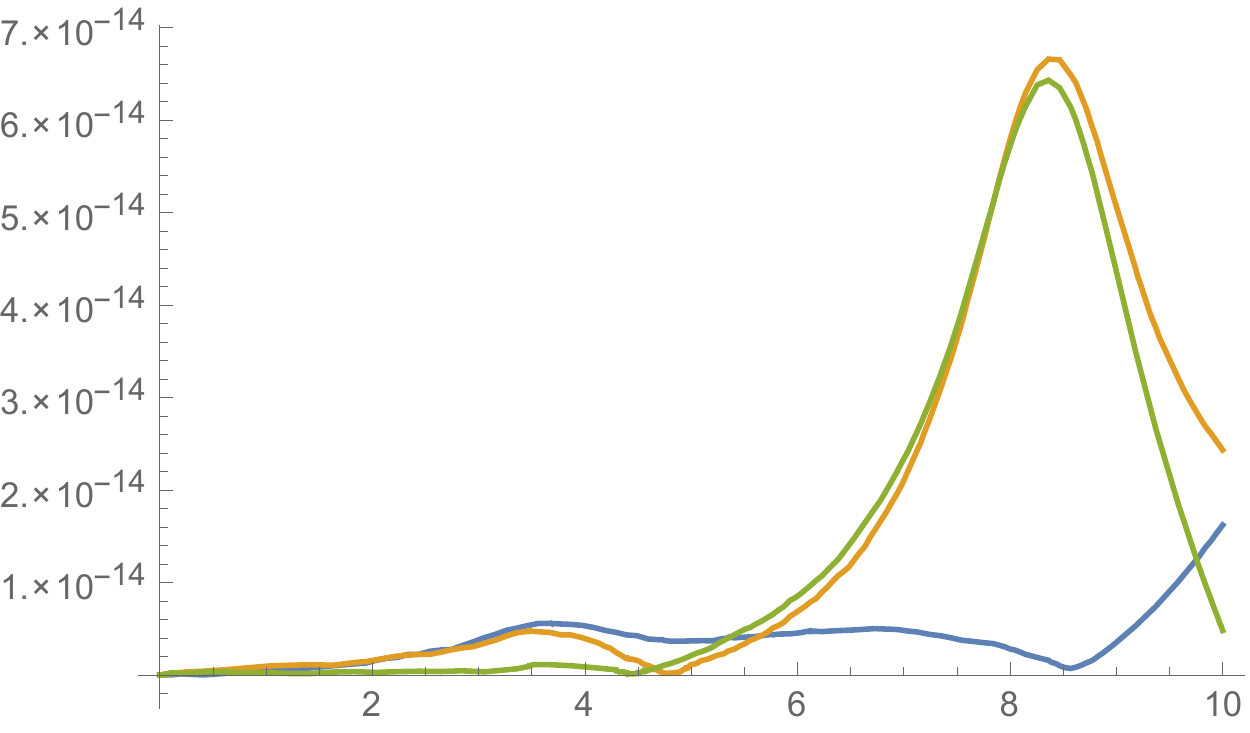} }}%
        \,
            \subfloat[normalisation check]{{\includegraphics[scale=0.33]{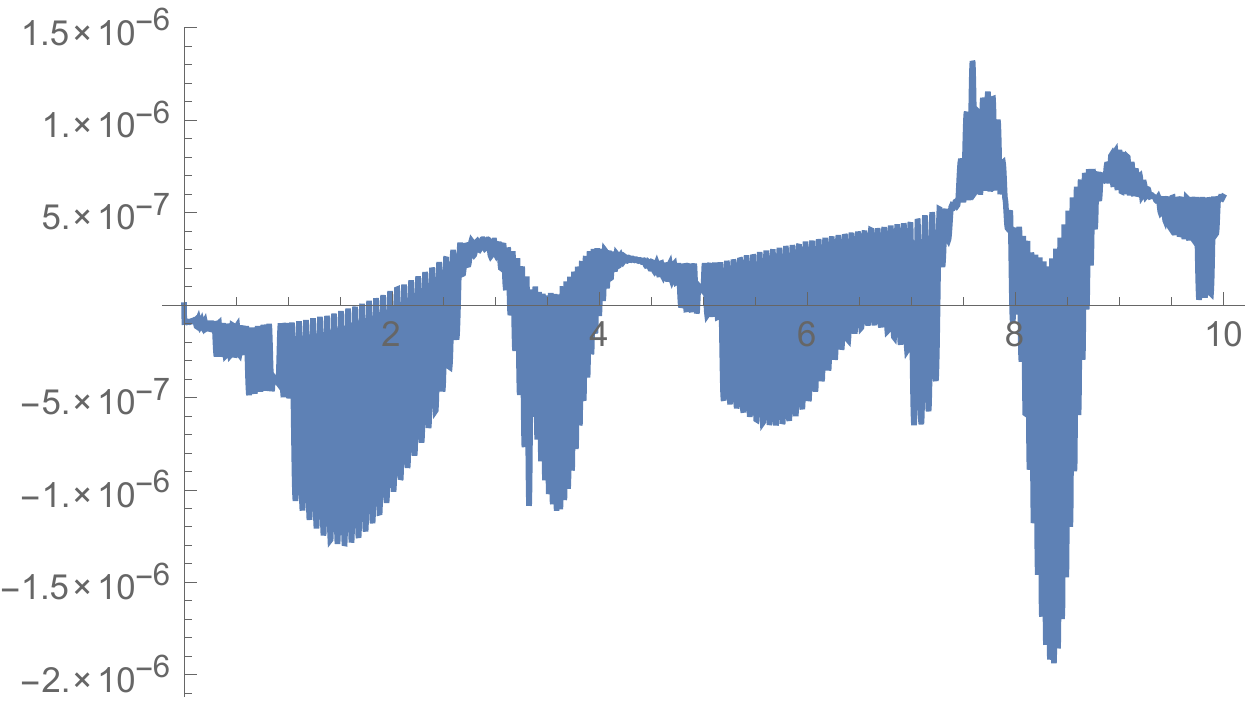} }}%
    \caption{Numerical solution to (\ref{gedeft7}) with initial conditions (\ref{gedef009ulp}), $t\in[0,10]$.
    }%
    \label{highden35}%
\end{figure}

\section{Two $M_2(\C)$ examples} \label{sectmatrix}

Take $A=M_2(\C)$ with calculus as described in Example~\ref{matcalcex}.  We give $\Omega^1$ a connection with
 $\nabla_{\Omega^1_A}(s^i)=0$, so $\sigma_A(s^i\tens s^j)=s^j\tens s^i$ and $\doublenabla(\sigma_A)=0$. 

\begin{example}
Take a $C^\infty(\R)$ valued inner product on 
$M=C^\infty(\R)\tens M_2(\C)$ by 
\[
\<c,\overline{e}\>=\mathrm{tr}(ce^*)\ .
\]
We first check the reality condition in Proposition~\ref{presip} for a vector field $K$, where
we set $K(s^i)=K_{i}\in M_2$,
\[
0=\<K((s^ia)^*)-K(s^ia)^*,\overline{1}\>=\<-K_{i'} a^*  - a^* K_i{}^*,\overline{1}\>
= - \mathrm{tr}(a^*(K_{i'} + K_i{}^*))
\]
for $i'\neq i$ and all $a\in M_2$, where we have used the centrality of $s^i$. 
We deduce that for real vector fields $K_1{}^*=-K_2$. 
For the divergence condition in Proposition~\ref{presip} we have
\begin{align*}
0 &= \<\big(ba+K(\extd a)+ab^*\big),\overline{1}\> =\<\big(ba+K(s^1\,[E_{12},a] + s^2\,[E_{21},a] )+ab^*\big),\overline{1}\>  \cr
&= \<\big(ba+K_1\,[E_{12},a] + K_2\,[E_{21},a] +ab^*\big),\overline{1}\> 
=\mathrm{tr}\big(a(b+b^*-[E_{12},K^1]-[E_{21},K_2]
)\big)
\end{align*}
and to solve this we set
\[
b=\tfrac12([E_{12},K_1]+[E_{21},K_2])
\]
which is Hermitian by the reality condition on $K$.

Then, by the comment after Corollary~\ref{ldtrax}, we would expect (subject to the uniqueness of solution) to have 
$K(K\tens\id)(\sigma_A-\id)=0$ on the interval if it is true for the initial condition. Now
\begin{align*}
K(K\tens\id)&(\sigma_A-\id)(s^1\tens s^2)=K(K\tens\id)(s^2\tens s^1-s^1\tens s^2) \\
&=K(K(s^2)s^1-K(s^1)s^2)=K(s^1K(s^2)-s^2K(s^1))=[K_1,K_2]
\end{align*}
so we require $[K_1,K_2]=0$, and by the reality condition this implies
$[K_1,K_1{}^*]=0$. 

Substituting the generators $s^i$ into the differential equation (\ref{gedef}) for $K$ gives
 \begin{align*}\label{gedef890}
 \tfrac{\partial K(s^i)}{\partial t}&= \tfrac{\partial K_i}{\partial t}   = K(bs^i) - bK(s^i)  - K(\extd K(s^i)) = -[b,K_i] -K(\extd K_i) \cr
 &=  [K_i,b]  -K_1\,[E_{12},K_i]  -K_2\,[E_{21},K_i]\ .
\end{align*}
Using $[K_1,K_2]=0$, this becomes
\[
 \tfrac{\partial K_i}{\partial t}  = 
 [K_i, (b+K_1E_{12}+K_2E_{21})]
\]
and using $b$ above we get
\begin{align} \label{achis}
 \tfrac{\partial K_i}{\partial t}  = \tfrac12\,
 [K_i, (E_{12}K_1+E_{21}K_2+K_1E_{12}+K_2E_{21})]\ .
\end{align}
From Proposition~\ref{prido} the equation $\nabla_M(m)=0$ gives
\[
\tfrac{\partial m}{\partial t} = -bm -K(\extd m)=-bm-K_1[E_{12},m]-K_2[E_{21},m]
\]
and using the value of $b$ above
\begin{align} \label{amm}
\tfrac{\partial m}{\partial t} =-\tfrac12([E_{12},K_1]+[E_{21},K_2])m-K_1[E_{12},m]-K_2[E_{21},m]\ .
\end{align}
We take the initial conditions
\begin{align} \label{startww}
K_1(0)=\left(\begin{array}{cc}1 & 0 \\0 & 2\end{array}\right) \ ,\ K_2(0)=- \left(\begin{array}{cc}1 & 0 \\0 & 2\end{array}\right)
 \ ,\ m(0)=\frac{1}{\sqrt{6}}\left(\begin{array}{cc}1 & 1 \\ 2 & 0\end{array}\right) 
\end{align}
and plot the numerical solutions for $K_1$ and $K_2$ in Figure~\ref{highden}, together with a check that $[K_1,K_2]=0$.

\begin{figure}
    \centering
    \subfloat[entries of $K_1(t)$]{{\includegraphics[scale=0.33]{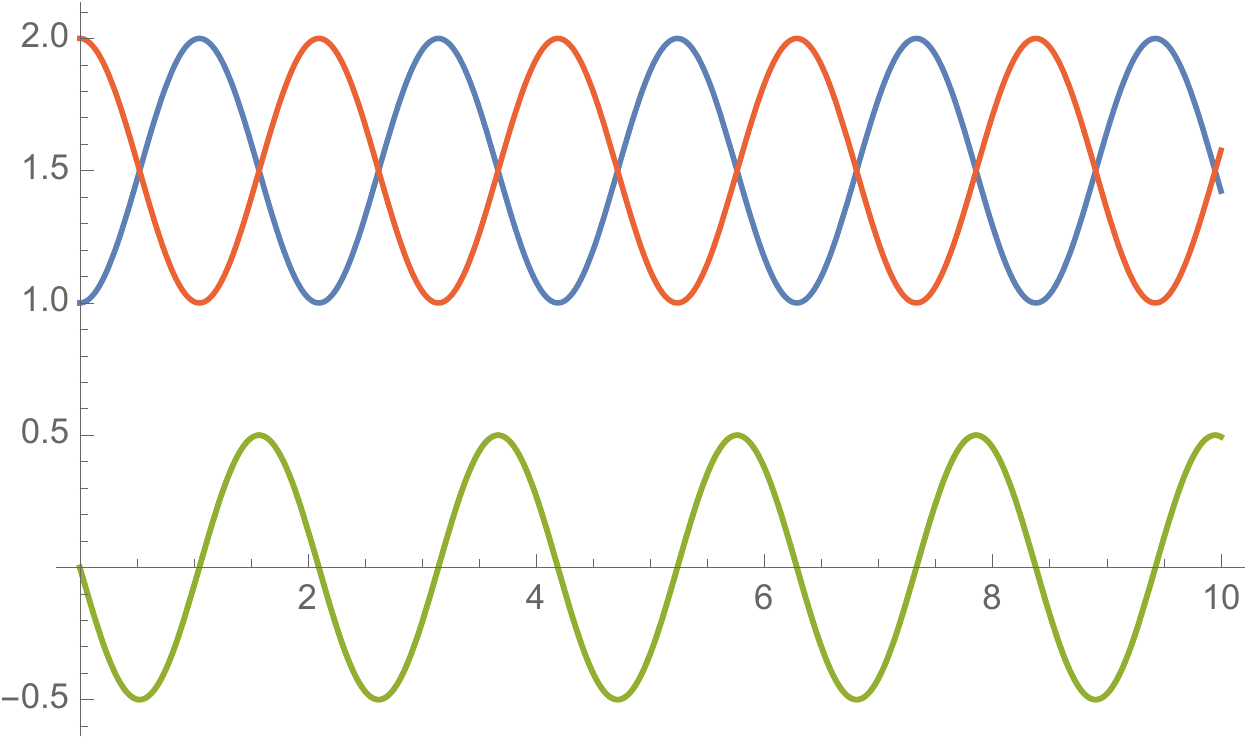}   } }%
    \,
    \subfloat[entries of $K_2(t)$]{{\includegraphics[scale=0.33]{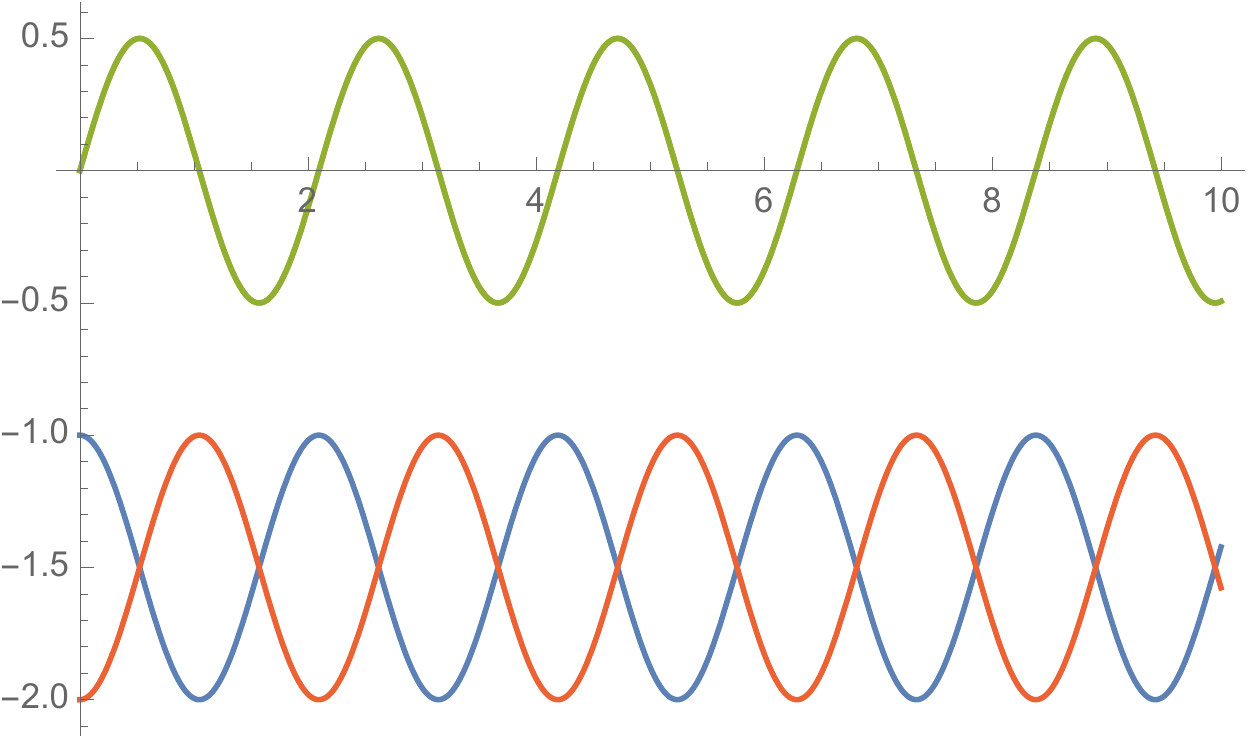} }}%
        \,
    \subfloat[entries of {$ [K_1(t),K_2(t)] $} ]{{\includegraphics[scale=0.33]{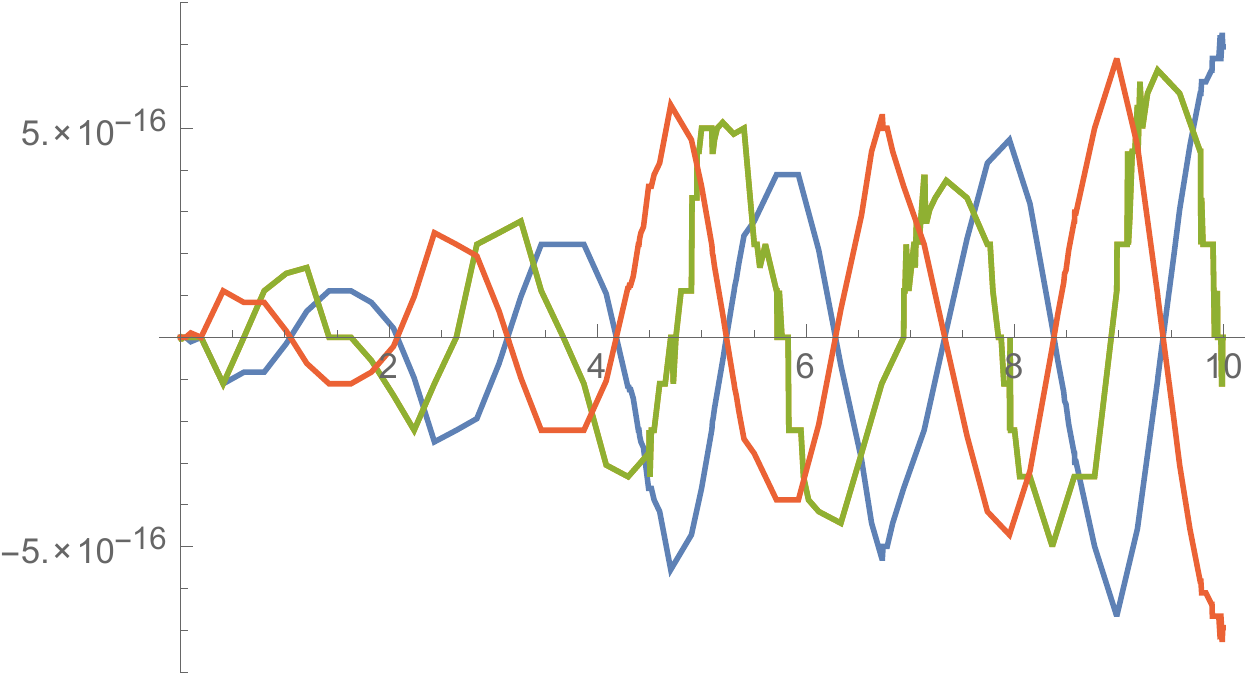} }}%
    \caption{Numerical solution to (\ref{achis}) with initial conditions (\ref{startww}) for $t\in[0,10]$.
    }%
    \label{highden}%
\end{figure}

By \cite{choimat} the pure states on $M_2(\C)$ are given by 
\begin{align} \label{m2st}
\phi\left(\begin{array}{cc}a & b \\c & d\end{array}\right)=
\left(\begin{array}{cc}\lambda & \mu\end{array}\right)
\left(\begin{array}{cc}a & b \\c & d\end{array}\right)
\left(\begin{array}{c}\lambda^* \\ \mu^*\end{array}\right)=|\lambda|^2a+\lambda\mu^*b+\lambda^*\mu c+|\mu|^2 d
\end{align}
for $\lambda,\mu\in\C$ with $|\lambda|^2+|\mu|^2=1$ chosen up to a multiple, so they are in 1-1 correspondence with $\mathbb{CP}^1$. 
If we set $x+\mathrm{i}\,y=\lambda\mu^*$ then the pure states can be written as
\begin{align} \label{m2s2t}
\phi\left(\begin{array}{cc}a & b \\c & d\end{array}\right)=
(\tfrac12-s)a+(x+\mathrm{i}\,y)b+(x-\mathrm{i}\,y) c+(\tfrac12+s) d
\end{align}
where $(s,x,y)\in\R^3$ with $x^2+y^2+s^2=\frac14$. The set of all normalised states is the points in the closed solid ball $x^2+y^2+s^2\le \frac14$.
We can find $(s,x,y)$ from $\phi$ by
\[
s=\tfrac{1}{2} \,\phi\big(\begin{smallmatrix} -1 & 0 \\ 0 & 1 \end{smallmatrix}\big)\ ,\ 
x=\tfrac{1}{2}\, \phi\big(\begin{smallmatrix} 0 & 1 \\ 1 & 0 \end{smallmatrix}\big)\ ,\ 
y=\tfrac{1}{2} \, \phi\big(\begin{smallmatrix} 0 & -\mathrm{i}  \\ \mathrm{i} & 0 \end{smallmatrix}\big)\ .
\]

 \begin{figure}
    \centering
    \subfloat[ Graph of $(s,x,y)$]{{\includegraphics[scale=0.33]{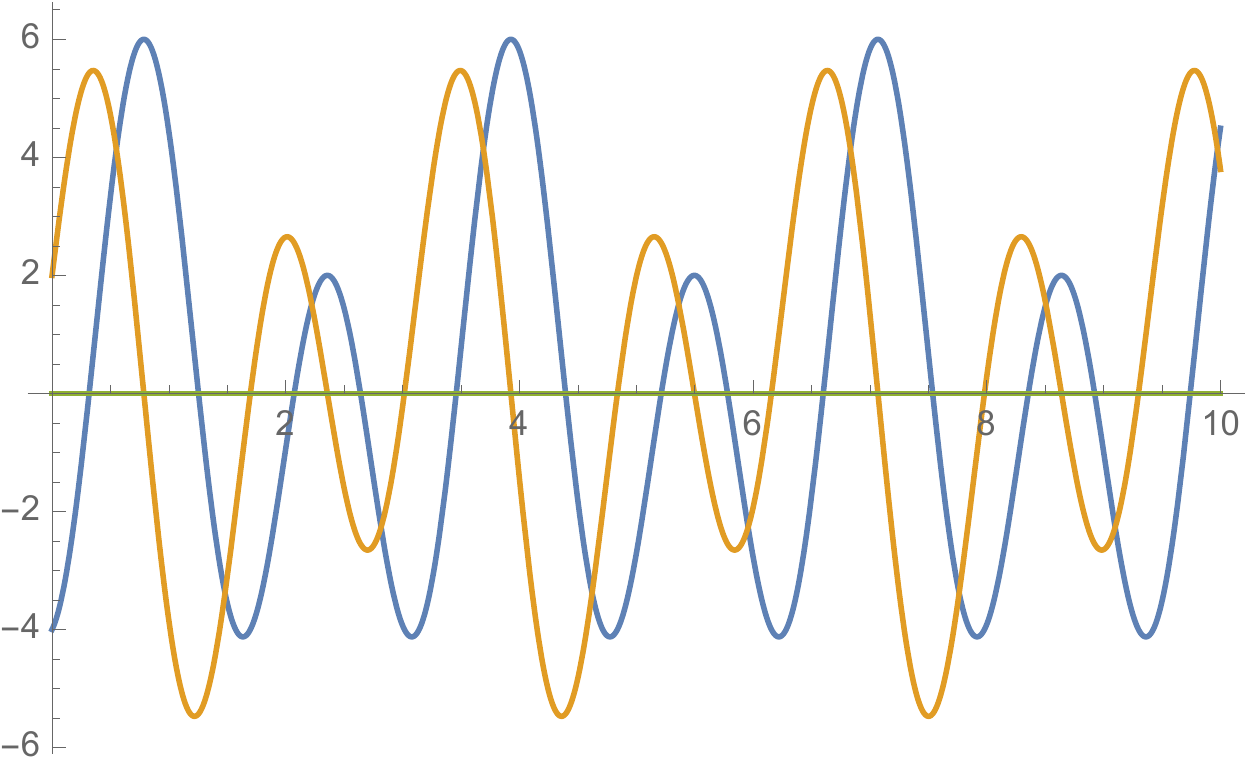}   } }%
    \,
    \subfloat[path in state space]{{\includegraphics[scale=0.25]{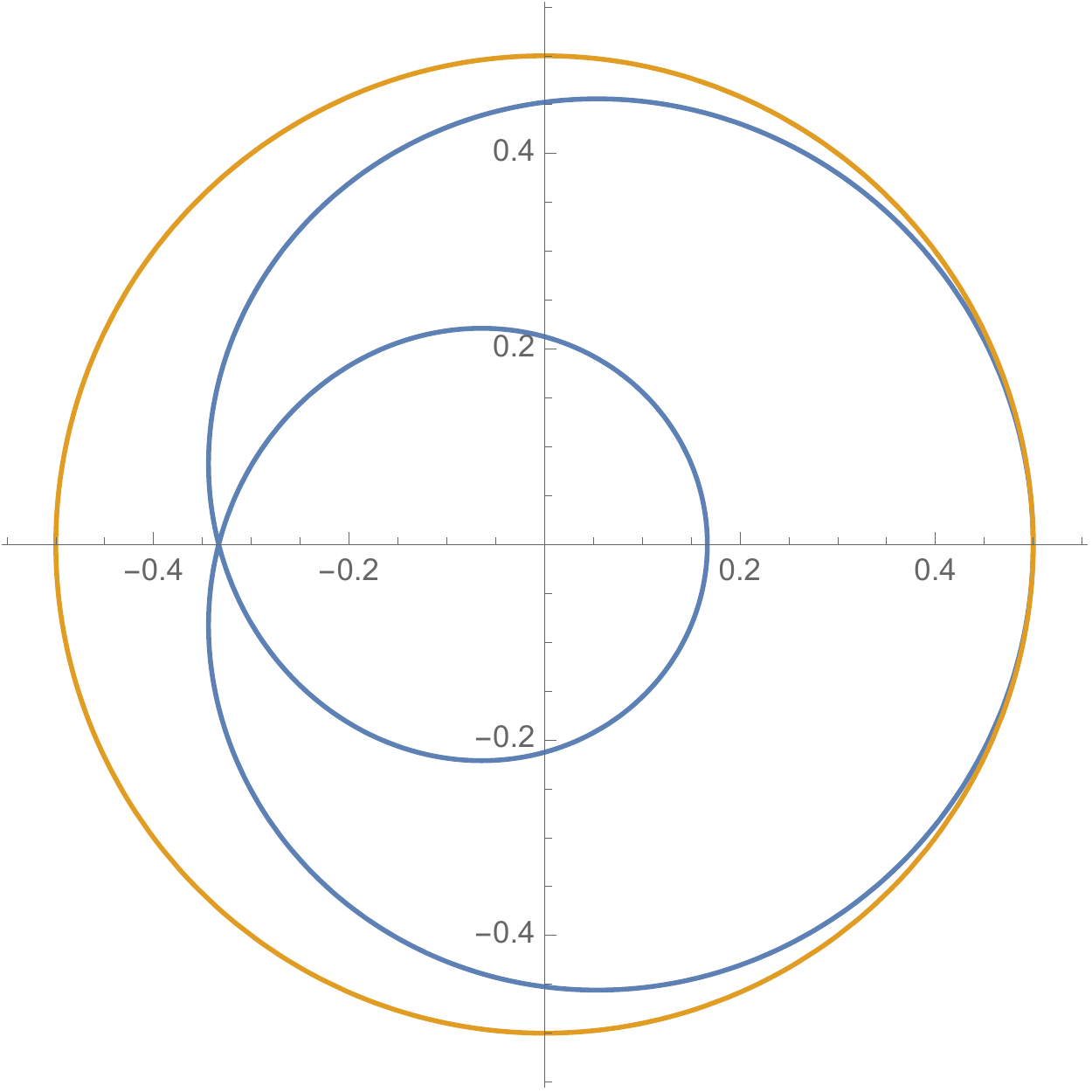} }}%
        \,
            \subfloat[normalisation check]{{\includegraphics[scale=0.33]{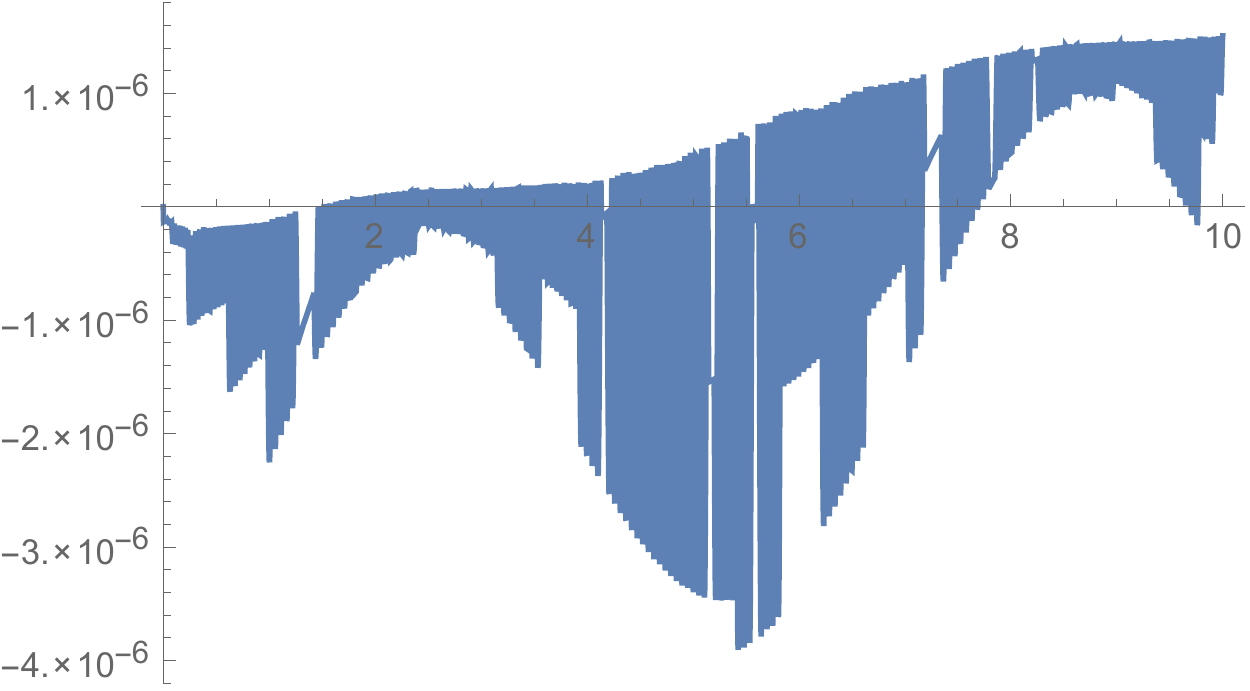} }}%
    \caption{Numerical solution to (\ref{amm}) with initial conditions (\ref{startww}), $t\in[0,10]$.
    }%
    \label{highden28}%
\end{figure}

Now $\phi(r)=\<m(t) r,\overline{m(t)}\>=\mathrm{trace}(mrm^*)$ for $r\in M_2(\C)$ gives a time dependent state on $M_2(\C)$, and the corresponding state is plotted in Figure~\ref{highden28}, with (a) a plot of the $(s,x,y)$ coordinates against time, and as $y=0$ for the initial conditions we can plot the path of the state in $(s,x)$ coordinates with the pure states given by the circle in (b). In (c) we plot 
$\phi(1)-1$ to check the normalisation of the state.  \hfill $\diamond$
 \end{example}

\medskip For a second example on the same algebra $A=M_2(\C)$ we consider another right module $N$ for $M_2(\C)$, and then set $M=C^\infty(\R)\tens N$. We then consider the equation
$\doublenabla(\sigma_M)=0$ and the corresponding flow, although this has no direct classical geodesic justification for this module. However, it does have the interesting property that it gives a familiar flow on the pure state space.

 \begin{example}
 For the algebra $A=M_2(\C)$ set the right $A$-module
 $N=\mathrm{Row}^2(\C)$ (the two dimensional row vectors), and then the inner product $\<n',\overline{n}\>=n'\,n^*\in\C$ gives the pure states, 
 as setting $n=(\lambda,\mu)$ (if non-zero) gives $\phi(r)=\<nr,\overline{n}\>/\<n,\overline{n}\>$ for $r\in M_2(\C)$ in
  (\ref{m2st}). Set $M=C^\infty(\R)\tens N$, and then
the possible bimodule maps $\sigma_M: M\tens_{A} \Omega^1_A\to \Omega^1(\R)\tens_{C^\infty(\R)} M$ are given by
\[
\sigma_M(w\tens s^i)=\extd t\tens Q_i\,w
\]
for $w\in \mathrm{Row}^2(\C)$ and $Q_i\in C^\infty(\R)$. Then for $f\in C^\infty(\R)$ we can take
\[
\nabla_M(f\,w)=\frac{\extd f}{\extd t} \,\extd t\tens w -\extd t\, f\tens w\big(Q_1\,E_{12}+Q_2\,E_{21}+Q_0\,I_2\big)
\]
for some $Q_0\in C^\infty(\R)$. 
We check the braid relation (\ref{outfd}) by calculating
\begin{align*}
(\sigma_{\Omega^1(\R)}\tens\id) (\id\tens\sigma_M)  (\sigma_M\tens\id)(f\,w\tens s^i\tens s^j) &=
(\sigma_{\Omega^1(\R)}\tens\id) (\id\tens\sigma_M)(\extd t\tens Q_i\,f\,w\tens s^j) \\
&= \extd t\tens \extd t\tens Q_jQ_i\,f\,w\ , \\
(\id\tens\sigma_M)    (\sigma_M\tens\id)  (\id\tens\sigma_A)(f\,w\tens s^i\tens s^j) &=
(\id\tens\sigma_M)    (\sigma_M\tens\id)(f\,w\tens s^j\tens s^i)\\
&= \extd t\tens \extd t\tens Q_jQ_i\,f\,w\ .
\end{align*}
Now Proposition~\ref{bcuow} shows that $\doublenabla(\sigma_M)$ is a bimodule map, so to show that $\doublenabla(\sigma_M)$ vanishes it is only necessary to show that it vanishes on generators.
Now
\begin{align*}
(\id\tens\sigma_M)
\nabla_{M\tens\Omega^1_A}( w\tens s^i)  &=  -\extd t\tens \sigma_M\big( w\big(Q_1\,E_{12}+Q_2\,E_{21}+Q_0\,I_2\big)\tens s^i\big)\\
&=   -\extd t\tens \extd t\tens  w\big(Q_1\,E_{12}+Q_2\,E_{21}+Q_0\,I_2\big)Q_i\ .
\end{align*}
and
\begin{align*}
\nabla_{\Omega^1(\R)\tens M}\sigma_M( w\tens s^i)
 &=\nabla_{\Omega^1(\R)\tens M}(Q_i\,\extd t\tens w) \\
 &= \frac{\extd Q_i}{\extd t} \,\extd t\tens \extd t\tens  w -\extd t\tens Q_i\,\extd t\tens w\big(Q_1\,E_{12}+Q_2\,E_{21}+Q_0\,I_2\big)
\end{align*}
  so $\doublenabla(\sigma_M)=0$ implies that $Q_1$and $Q_2$
 are constant and $Q_0$ is arbitrary. 
Now set $m=(\lambda ,\mu )(t)$ and then $\nabla_M(m)=0$ gives
\begin{align*}
\big(\frac{\extd \lambda }{\extd t} , \frac{\extd \mu }{\extd t}\big)=\big( Q_0\,\lambda +Q_2\,\mu ,
Q_0\,\mu  + Q_1\,\lambda  \big)
\end{align*}
Putting $z=\lambda /\mu $ we get $\frac{\extd z}{\extd t}=Q_2-Q_1\,z^2$. In terms of the action of
$SL_2(\C)$ on the Riemann sphere $\C_\infty$ of pure states by M\"obius transformations
\[
\left(\begin{array}{cc}a & b \\c & d\end{array}\right)\la z=\frac{az+b}{cz+d}
\]
the time action is given by the one parameter group given by 
$\exp\big(t\big(\begin{smallmatrix} 0 & Q_2  \\ Q_1 & 0 \end{smallmatrix}\big)\big)$. 

We have not yet checked whether the connection preserves the inner product, but rather have allowed a possibly varying normalisation for the positive map. It is easy to check that the condition for preserving the inner product is $Q_0{}^*=-Q_0$ and $Q_2=-Q_1{}^*$, and that the normalisation of the positive maps is then constant in time.
\hfill $\diamond$
  \end{example}

\section{Appendix: Numerical calculations}
The Mathematica code for the numerical simulation and drawing Figure~\ref{highden35} (a) in Section~\ref{z3exsect} follows. Here $K_+(i)$ is kp$i$, 
$K_-(i)$ is km$i$, and $m(i)$ is m$i$ for $i\in\{0,1,2\}$.

    \smallskip\noindent
{\footnotesize{ 
sspp = NDSolve[\{
   kp1'[t] == kp1[t] (kp0[t] + km2[t] - kp2[t] - km1[t])/2,
   kp2'[t] == kp2[t] (kp1[t] + km0[t] - kp0[t] - km2[t])/2, 
   kp0'[t] == kp0[t] (kp2[t] + km1[t] - kp1[t] - km0[t])/2,
   km1'[t] == km1[t] (kp0[t] + km2[t] - kp1[t] - km0[t])/2,
   km2'[t] == km2[t] (kp1[t] + km0[t] - kp2[t] - km1[t])/2, 
   km0'[t] == km0[t] (kp2[t] + km1[t] - kp0[t] - km2[t])/2,
   km1[0] == 1, km2[0] == Exp[2 I], km0[0] == Exp[3 I], 
   kp1[0] == -Exp[-3 I], kp2[0] == -1, kp0[0] == -Exp[-2 I], 
   m0'[t] == - m0[t] (-kp0[t] + kp1[t] - km0[t] + km2[t])/2 - 
     kp0[t] (m0[t] - m2[t]) - km0[t] (m0[t] - m1[t]),
   m1'[t] == -m1[t] (-kp1[t] + kp2[t] - km1[t] + km0[t])/2 - 
     kp1[t] (m1[t] - m0[t]) - km1[t] (m1[t] - m2[t]),
   m2'[t] == -m2[t] (-kp2[t] + kp0[t] - km2[t] + km1[t])/2 - 
     kp2[t] (m2[t] - m1[t]) - km2[t] (m2[t] - m0[t]), 
   m0[0] == 1/Sqrt[2], m1[0] == 0, m2[0] == 1/Sqrt[2]
   \}, \{km1, km2, km0, kp1, kp2, kp0, m0, m1, m2\}, \{t, 0, 10\}]
   }}
   
   \medskip\noindent
   {\footnotesize{ 
   Plot[Evaluate[\{(Abs[m0[t]]$\wedge$2), (Abs[m0[t]]$\wedge$2 + Abs[m1[t]]$\wedge$2), 
    Abs[m0[t]]$\wedge$2 + Abs[m1[t]]$\wedge$2 + Abs[m2[t]]$\wedge$2\} /. sspp], \{t, 0, 10\}, 
 PlotRange -$>$ All]
   }}

\medskip
\noindent
The Mathematica code for the numerical simulation and drawing Figure~\ref{highden28} (b) in Section~\ref{sectmatrix} follows: Here $m$ is
\{\{ma,mb\},\{mc,md\}\}, $K_1$ is \{\{a1,b1\},\{c1,d1\}\}, and similarly for $K_2$. 

    \smallskip\noindent
{\footnotesize{ 
sst = NDSolve[\{
   Derivative[1][a1][
     t] == (-c1[t] (a1[t] + d1[t]) + b1[t] (a2[t] + d2[t]))/2, 
   Derivative[1][b1][t] == (a1[t]$\wedge$2 - d1[t]$\wedge$2)/2, 
   Derivative[1][c1][t] == ((-a1[t] + d1[t]) (a2[t] + d2[t]))/2, 
   Derivative[1][d1][
     t] == (c1[t] (a1[t] + d1[t]) - b1[t] (a2[t] + d2[t]))/2, 
   Derivative[1][a2][
     t] == (-c2[t] (a1[t] + d1[t]) + b2[t] (a2[t] + d2[t]))/2,
   Derivative[1][b2][t] == ((a1[t] + d1[t]) (a2[t] - d2[t]))/2,
   Derivative[1][c2][t] == (-a2[t]$\wedge$2 + d2[t]$\wedge$2)/2,
   Derivative[1][d2][
     t] == (c2[t] (a1[t] + d1[t]) - b2[t] (a2[t] + d2[t]))/2, 
   Derivative[1][ma][
     t] == -(c1[t] ma[t] - 2 a2[t] mb[t] + a1[t] mc[t] + d1[t] mc[t] +
         b2[t] (ma[t] - 2 md[t]))/2,
   Derivative[1][mb][
     t] == -(b2[t] mb[t] + c1[t] mb[t] - 2 b1[t] mc[t] + d1[t] md[t] +
         a1[t] (-2 ma[t] + md[t]))/2,
   Derivative[1][mc][
     t] == -(a2[t] ma[t] - 2 c2[t] mb[t] + b2[t] mc[t] + c1[t] mc[t] +
         d2[t] (ma[t] - 2 md[t]))/2,
   Derivative[1][md][
     t] == -(a2[t] mb[t] + d2[t] mb[t] - 2 d1[t] mc[t] + b2[t] md[t] +
         c1[t] (-2 ma[t] + md[t]))/2, a1[0] == 1, b1[0] == 0, 
   c1[0] == 0, d1[0] == 2, a2[0] == -1, b2[0] == 0, c2[0] == 0, 
   d2[0] == -2, ma[0] == 1/Sqrt[6], mb[0] == 1/Sqrt[6], 
   mc[0] == 2/Sqrt[6], md[0] == 0
   \}, \{a1, a2, b1, b2, c1, c2, d1, d2, ma, mb, mc, md\}, \{t, 0, 10\}]
   }

    \medskip\noindent
{\footnotesize{ 
ParametricPlot[\{Evaluate[\{-Conjugate[ma[t]] ma[t] + 
       Conjugate[mb[t]] mb[t] - Conjugate[mc[t]] mc[t] + 
       Conjugate[md[t]] md[t], 
      Conjugate[mb[t]] ma[t] + Conjugate[ma[t]] mb[t] + 
       Conjugate[md[t]] mc[t] + Conjugate[mc[t]] md[t]\} /. sst]/
   2, \{Cos[t]/2, Sin[t]/2\}\}, \{t, 0, 10\}]
   }


\begin{thebibliography}{ggghhh}

\bibitem{AppSto}
D.\ Applebaum,  Stochastic Evolution of Yang-Mills Connections on the Noncommutative Two-Torus, Letters in Mathematical Physics 16 (1988) 93-99.

\bibitem{AscSch}
P.\ Aschieri and  P.\ Schupp, Vector fields on Quantum Groups, Int.\ J.  Mod.\ Phys.\ A, 11 (1996) 1077-1100

\bibitem{bbsheaf} E.J.\ Beggs and  T.\ Brzezi\'nski, 
The Serre spectral sequence of a noncommutative fibration for de Rham cohomology, 
Acta Math. 195 (2005) 155--196


\bibitem{BMriem}
E.J.\ Beggs and  S.\ Majid, *-compatible connections in noncommutative Riemannian geometry, J.\ Geom.\ Phys.\ 61 (2011)  95--124

\bibitem{BegMa:bia} E.J. Beggs and S. Majid, Quantum Bianchi identities via DG categories, J. Geom. Phys.\ 124 (2018) 350--370

\bibitem{BlackKth}
B.\ Blackadar, $K$-theory for operator algebras, MSRI Publications, Berkeley, 1986.

\bibitem{Borowiec}
A.\ Borowiec, Vector fields and differential operators: noncommutative case, Czech. J.\ Phys.\  47  (1997) 1093--1100

\bibitem{BDMS}
K. Bresser, F. M\"uller-Hoissen, A. Dimakis and  A. Sitarz, Noncommutative 
geometry of finite groups. J.\ Phys. A, 29 (1996) 2705--2735

\bibitem{choimat}
 M-D.\ Choi, Completely positive linear maps on complex matrices, Lin. Algebra Applic. 10 (1975) 285--290

\bibitem{Connes}
A. Connes,   Noncommutative Geometry, 
Academic Press, Inc., San Diego, CA, 1994

\bibitem{conHig} A.\ Connes  and  N.\ Higson, D\'eformations, morphismes asymptotiques et K-th\'eorie bivariante, C.R.\ Acad.\ Sci.\ Paris S\'er.\  I Math.\ 311 (1990) 101--106

\bibitem{DadAsy} M.\ D\~ad\~arlat, Shape theory and asymptotic morphisms for $C^*$-algebras, Duke Math.\ J.\ 73 (1994)  687--711

\bibitem{DVMass}
M.\ Dubois-Violette and  T.\ Masson, On the first-order operators in bimodules, 
Lett.\ Math.\ Phys.\ 37 (1996) 467--474

\bibitem{DVMic}
M. Dubois-Violette and  P.W.\ Michor, Connections on central bimodules in 
noncommutative differential geometry, J.\ Geom.\ Phys.\ 20 (1996) 218--232

\bibitem{FioMad}
G.\ Fiore and  J.\ Madore, Leibniz rules and reality conditions, Eur.\ Phys.\ J.\ C Part.\ Fields 17 (2000) 359--366

\bibitem{MO-sebgo}
S.\ Goette,
Surfaces extending modified geodesic paths, URL (version: 2019-03-28):
https://mathoverflow.net/q/326559

\bibitem{HawRad}
S.W.\ Hawking, Particle creation by black holes, Commun.\ Math.\ Phys.\ 43, 199--220 (1975)

\bibitem{JarLle}
P. Jara and D. Llena, Lie bracket of vector fields in noncommutative geometry,
Czech. J.\ Phys.\ 53 (2003) 743--758

\bibitem{Lance} E.C.\ Lance, Hilbert $C^*$-modules, A toolkit for operator algebraists, LMS.\ Lecture Note Series 210, C.U.P.\ 1995

\bibitem{LychBraid}  V.\ Lychagin, Calculus and Quantizations Over Hopf Algebras, Acta Appl.\ Math.\ 51 (1998) 303--352

\bibitem{Madore}
J.\ Madore, An introduction to noncommutative differential geometry
 and its physical applications, LMS Lecture Note Series, 257, C.U.P.  1999. 
 
 \bibitem{Ma:cla} S. Majid, Classification of bicovariant differential calculi, J. Geom. Phys. 25 (1998) 119--140
 
\bibitem{ManThom} V.\ Manuilov and  K.\ Thomsen, Shape theory and extensions of $C^*$-algebras, J.\ London Math.\ Soc.  84 (2011) 183--203


\bibitem{Mourad}
J.\ Mourad, Linear connections in noncommutative geometry, Class.\ Quant. 
Grav.\ 12 (1995)  965--974


\bibitem{OneillRiem} B.\ O'Neill, Semi-Riemannian geometry with applications to relativity, Pure and Applied Mathematics, Academic Press, 1983



















\end{thebibliography}
\end{document}